\documentclass[aap,preprint]{imsart}

\RequirePackage{amsthm,amsmath,amsfonts,amssymb,mathrsfs,stmaryrd,stackrel}
\RequirePackage[numbers]{natbib}
\RequirePackage[colorlinks,citecolor=blue,urlcolor=blue]{hyperref}
\RequirePackage{graphicx,color,epsfig,subfigure}
\RequirePackage{bm,bbm}

\startlocaldefs

\newtheorem{theorem}{Theorem}[section]
\newtheorem{proposition}{Proposition}[section]

\newtheorem{corollary}[theorem]{Corollary}
\theoremstyle{remark}

\newtheorem{definition}[theorem]{Definition}

\newcommand{\Prob}[1]{\mathbb{P}\left[#1\right]}

\newcommand{\Exp}[1]{\mathbb{E}\left[#1\right]}

\newcommand{\dd}{\mathrm{d}}

\endlocaldefs

\begin{document}

\begin{frontmatter}
\title{Characterization of blowups via time change in a mean-field neural network}
\runtitle{Blowing up via time change}

\begin{aug}

\author[A,B]{\fnms{Thibaud} \snm{Taillefumier}\ead[label=e1]{ttaillef@austin.utexas.edu}}
\and
\author[A,B]{\fnms{Phillip} \snm{Whitman}\ead[label=e3]{second@somewhere.com}}
\address[A]{Department of Mathematics,
University of Texas, Austin,
}

\address[B]{Department of Neuroscience,
University of Texas, Austin,
}
\end{aug}

\begin{abstract}
Idealized networks of integrate-and-fire neurons with impulse-like interactions obey McKean-Vlasov diffusion equations in the mean-field limit. 
These equations are prone to blowups: for a strong enough interaction coupling, the mean-field rate of interaction diverges in finite time with a finite fraction of neurons spiking simultaneously, thereby marking a macroscopic synchronous event.
Characterizing these blowup singularities analytically is the key to understanding the emergence and persistence of spiking synchrony in mean-field neural models. 
However, such a resolution is hindered by the first-passage nature of the mean-field interaction in classically considered dynamics.
Here, we introduce a delayed Poissonian variation of the classical integrate-and-fire dynamics for which blowups are analytically well defined in the mean-field limit.
Albeit fundamentally nonlinear, we show that this delayed Poissonian dynamics can be transformed into a noninteracting linear dynamics via a deterministic time change.
We specify this time change as the solution of a nonlinear, delayed integral equation via renewal analysis of first-passage problems.
This formulation also reveals that the fraction of simultaneously spiking neurons can be determined via a self-consistent, probability-conservation principle about the time-changed linear dynamics.
We utilize the proposed framework in a companion paper to show analytically the existence of singular mean-field dynamics with sustained synchrony for large enough interaction coupling.
\end{abstract}



\end{frontmatter}


\section{Introduction}


\subsection{Background} 

This work introduces neural network models for which the emergence of synchrony can be studied analytically in the idealized, mean-field limit of infinite-size networks.
By synchrony, we refer to the possibility that a finite fraction of the network's neurons simultaneously spikes.
Dynamics exhibiting such synchrony can serve as models to study the maintenance of precise temporal information in neural networks \cite{Panzeri:2010,Kasabov:2010,Brette:2015}.
The maintenance of precise temporal information in the face of neural noise remains a debated issue from an experimental and computational perspective.
Mathematical approaches to understand synchrony involve making simplifying assumptions about the individual neuronal processing as well as about the network supporting their interactions.

Integrate-and-fire neurons \cite{Lapicque:1907,Knight:1972} constitute perhaps the simplest class of models susceptible to displaying synchrony \cite{Taillefumier:2013uq}.
In integrate-and-fire models, the internal state of a neuron $i$ is modeled as a continuous-time diffusive process $X_{i,t}$, whose dynamics stochastically integrates past neural interactions.
Spiking times are then defined as first-passage times of this diffusive process to a spiking boundary $L$. 
Upon spiking, the process $X_{i,t}$ resets to a base real value $\Lambda$.
In other words, $X_{i,t^+}=\Lambda$ whenever neuron $i$ spikes at time $t$.
There is no loss of generality in assuming that  $L=0$ and $\Lambda>0$, so that $X_{i,t}$ has nonnegative state space.
Moreover, classical integrate-and-fire models assume that $X_{i,t}$ follows a Wiener diffusive dynamics with negative drift $-\nu<0$ \cite{Carrillo:2015}.
Drifted Wiener processes are the simplest diffusive dynamics for which spikes occur in finite time with probability one, even in the absence of interactions.

A key feature of integrate-and-fire models is that they allow for the occurrence of synchronous spiking events.
This is most conveniently seen by considering a finite neural network with instantaneous, homogeneous, impulse-like excitatory interactions.
For such interactions, if neuron $i$ spikes at time $t$, downstream neurons $j \neq i$ instantaneously update their internal states according to $X_{j,t^+} = X_{j,t^+} - w$, where $ w>0$ is the size of the impulse-like interaction.
Thus, the spiking of neuron $i$ causes the states of all other neurons to move toward the zero spiking boundary, leading to two possible outcomes for downstream neuron $j$: 
either $X_{j,t^+}>w$ and the update merely hastens the next spiking time of neuron $j$, or  $X_{j,t^+} \leq w$ and the interaction causes neuron $j$ to spike in synchrony with $i$ \cite{Faugeras:2011,Taillefumier:2012uq}.
The latter synchronous spiking events occurs with finite probability, as we generically have $\Prob{X_{j,t} \in (0,w]}>0$ for regular diffusion processes.
In turn, the synchronous spiking of downstream neuron $j$ can trigger additional synchronous spiking events in the network, via branching processes referred to as spiking avalanches.
Spiking avalanches are well-defined under the modeling assumptions that neurons transiently enter a post-spiking refractory state and always exit this refractory state by reseting to $\Lambda$  \cite{Taillefumier:2014fk}.
Under such assumptions, neurons can spike at most once within an avalanche and synchronously spiking neurons can be distinguished according to their generation number \cite{Delarue:2015}.

Tellingly, the finite probability to observe a spiking avalanche is maintained in certain simplifying limit, such as the thermodynamic mean-field limit \cite{Amari:1975aa, Faugeras:2009,Touboul:2014,Renart:2004}.
For a homogeneous, excitatory, integrate-and-fire network, the thermodynamic mean-field limit considers a network of $N$ exchangeable neurons in the infinite-size limit, $N\to \infty$, with vanishingly small impulse size $w_N = \lambda/N \to 0$, where $\lambda$ is a parameter quantifying the interaction coupling.
In this mean-field limit, individual neurons only interact with one another via a deterministic population-averaged firing rate $f(t)$ \cite{Caceres:2011,Carrillo:2013,Carrillo:2015}.
Specifically, the dynamics of a representative process $X_t$ obeys a nonlinear partial differential equation (PDE) of the McKean-Vlasov type
\begin{eqnarray}\label{eq:classicalMV}
\partial_t p = \big(\nu + \lambda f(t) \big)  \partial_x p + \partial_x^2 p/2  + f(t^-) \delta_\Lambda \, ,  \quad  \mathrm{with} \quad p(t,x) \, \dd x =\Prob{X_t \in \dd x} \, ,
\end{eqnarray}
and where the effective drift features the firing rate $f(t)$.
The nonlinearity of the above equation stems from the conservation of probability, which equates $f(t)$ with a boundary flux of probability:
\begin{eqnarray}\label{eq:classicalFlux}
 f(t) = \partial_x p(t,0)/2 \, .
\end{eqnarray}
In the following, we refer to equation \eqref{eq:classicalMV} and \eqref{eq:classicalFlux} as the classical McKean-Vlasov (cMV) equations and to the underlying dynamics supporting these equations as the classical mean-field (cMF) dynamics.
Within the setting of cMF dynamics, a blowup occurs at time $T_0$ if the spiking rate diverges when $t \to T_0^-$ and a synchronous event happens if a  fraction of the neurons $\pi_0>0$ synchronously spikes in $T_0$.


\subsection{Motivation} 

Following on seminal computational work in \cite{Brunel:1999aa,Brunel:2000aa}, the cMF dynamics was first investigated in a PDE setting by C\'aceres {\it et al.} \cite{Caceres:2011}, who established the occurrence of blowups.
The existence and regularity of solutions to \eqref{eq:classicalMV} and \eqref{eq:classicalFlux} have been considered from the standpoint of stochastic analysis by several authors \cite{Delarue:2015,Delarue:2015b,Hambly:2019,Nadtochiy:2019,Nadtochiy:2020}.
These authors combined results from the theory of interacting-particle systems \cite{Liggett:1985,Sznitman:1989} and of the convergence of probability measures \cite{Billingsley:2013} to establish criteria for the existence of global solutions \cite{Delarue:2015,Delarue:2015b} and to classify the type of singularities displayed by these solutions \cite{Hambly:2019,Nadtochiy:2019,Nadtochiy:2020}.
However, the analytical characterization of blowup singularities have proven rather challenging.
Here, we propose a modified interacting-particle system with Poisson-like attributes that is also prone to blowup, the so-called delayed Poissonian mean-field (dPMF) model.
By contrast with \cite{Delarue:2015b} and in line with \cite{Caceres:2011,Carrillo:2013}, we only conjecture that the propagation of chaos holds to 
motivate the form of the corresponding mean-field PDE problem.
This conjecture is numerically supported in the weak interaction regime  $\lambda < \Lambda$ and for the strong interaction regime $\lambda > \Lambda$.
The interest of the proposed framework lies in introducing a mean-field model where blowups, including full blowup whereby a finite fraction of neurons fires synchronously, can be studied analytically.
In particular, we utilize this framework in \cite{TTLS} to show the existence of global solutions defined on the whole real lines with an infinite but countable number of blowups for large interaction parameters $\lambda \gg \Lambda$.



\subsection{Approach} 

The crux of our approach is the introduction of an analytically tractable neural-network model that is closely related to the cMF model, the so-called delayed Poissonian mean-field (dPMF) dynamics.
dPMF dynamics are derived from the classical ones by considering that $(a)$ neurons are driven by noisy inputs with Poisson-like attributes and that $(b)$ neurons exhibit a post-spiking refractory period.
Concretely, assumption $(a)$ corresponds to approximating the counting process registering neuronal inputs in the thermodynamic limit by a Gaussian Markov process with time-dependent drift $-\big(\nu + \lambda f(t)\big)$ and unit Fano factor, i.e., with variance and drift of identical magnitude.
At the same time, assumption $(b)$ corresponds to enforcing that neurons remain in a noninteracting, inactive state for a duration $\epsilon$ after reaching the zero spiking threshold and before reseting in $\Lambda$.
In addition of being relevant from a modeling standpoint, the inclusion of a finite refractory period $\epsilon$ allows for the unambiguous definition of dPMF dynamics during synchrony.
Specifically, refractory period enforces that every neuron engaging in an instantaneous spiking avalanche at time $t$ spikes only once and resets in $\Lambda$  at time $t+\epsilon$.
Overall, the delayed Poissonian version of the nonlinear McKean-Vlasov dynamics \eqref{eq:classicalMV} reads
\begin{eqnarray}
\partial_t p = \big(\nu + \lambda f(t) \big) \left( \partial_x p + \partial_x^2 p/2 \right) + f(t-\epsilon) \delta_\Lambda \, ,
\nonumber
\end{eqnarray}
where we will see that the conservation of probability imposes that 
\begin{eqnarray}\label{eq:cPMF}
f(t) = \frac{\nu \partial_x p(t-\epsilon,0)}{2-\lambda  \partial_x p(t,0)} \, .
\end{eqnarray}
The above relation directly indicates the criterion for blowups in dPMF dynamics:
blowups occur whenever $\partial_x p(t,0)/2$, the instantaneous flux through the absorbing boundary, reaches the value $1/\lambda$.
In other words, blowups emerge at finite boundary flux, which allows for the continuous maintenance of the absorbing boundary condition: $p(t,0)=0$.
This is by contrast with cMF dynamics for which blowups involve diverging fluxes at times $T_0$, for which the absorbing boundary condition must  locally fail: $p(T_0,0) \geq 1/\lambda$ \cite{Nadtochiy:2019,Nadtochiy:2020}.
Such singular behavior is a major hurdle to elucidating blowup analytically in cMF dynamics. 
The expected regularized behavior of dPMF dynamics during blowups is the primary motivation for their introduction.

Ideally, the  PDE problem that defines dPMF dynamics shall be established as the mean-field limit of the corresponding finite-size interacting-particle system.
The present work only conjectures that such a mean-field limit holds, which is supported by numerical simulations (except possibly for interaction parameter $\lambda \simeq \Lambda$).
Then, the core idea of our approach is to solve the PDE problem defining dPMF dynamics by formally introducing the time change
\begin{eqnarray}\label{eq:sigmaTime_int}
\sigma = \Phi(t) = \nu t + \lambda F(t) \, , \quad \mathrm{with} \quad  F(t)=\int_0^t f(s) \, \dd s \, ,
\end{eqnarray}
which is a smooth increasing function in the absence of blowups.
Due to the Poissonian attributes of the neuronal drives, the time change $\Phi$ can serve to parametrize the dPMF dynamics of a representative process as $X_t=Y_{\Phi(t)}$, where $Y_\sigma$ is a process obeying a linear, noninteracting dynamics.
The dynamics of $Y_\sigma$ is that of a Wiener process absorbed in zero, with constant negative unit drift and with reset in $\Lambda$, but with time-inhomogeneous refractory period specified via a $\Phi$-dependent delay function $\sigma \mapsto \eta[\Phi](\sigma)$.
In the following, we will refer to $\eta$ as the backward delay function associated to $\Phi$.
Concretely, this means that assuming the backward-delay function $\eta$ known, the transition kernel of $Y_\sigma$ denoted by  $(\sigma,x) \mapsto q(\sigma,x)$ satisfies the time-changed PDE problem 
\begin{eqnarray}\label{eq:qPDE_int}
\partial_\sigma q &=& \partial_x q +\frac{1}{2} \partial^2_{x} q  + \frac{\dd}{\dd \sigma} [G(\sigma-\eta(\sigma))] \delta_{\Lambda}  \, , 
\end{eqnarray}
with absorbing and conservation conditions respectively given by 
\begin{eqnarray}\label{eq:abscons_int}
q(\sigma,0)=0 \quad \mathrm{and} \quad \partial_\sigma G(\sigma)=\partial_x q(\sigma ,0) /2 \, .
\end{eqnarray}
In equations \eqref{eq:qPDE_int} and \eqref {eq:abscons_int}, $G$ denotes the $\eta$-dependent cumulative flux of $Y_\sigma$ through the zero threshold.
By definition of the time change $\Phi$, $G$ is related to the cumulative flux $F$ via $F=G \circ \Phi$.
Moreover, in the absence of blowups, the functional dependence of $\eta$ on the time change $\Phi$ is given by 
\begin{eqnarray}
\eta(\sigma) = \sigma - \Phi(\Psi(\sigma) -\epsilon) \, .
\nonumber
\end{eqnarray}
where $\Psi=\Phi^{-1}$ refers to the inverse time change of $\Phi$.
Thus, $G=G[\Phi]$ actually depends on $\Phi$ via $\eta$, which motivates considering equation \eqref{eq:sigmaTime_int}  as a self-consistent equation specifying admissible time changes:
\begin{eqnarray}\label{eq:selfcons_int}
\Phi(t) = \nu t + \lambda G[\Phi](\Phi(t)) \, .
\end{eqnarray}
Our approach then elaborates on the fact that in the absence of blowups, dPMF dynamics are fully parametrized by the time-change function that uniquely solves equation \eqref{eq:selfcons_int} for some reasonable initial conditions.
Given such a solution $\Phi$, the transition kernel of a representative dPMF dynamics $X_t$ is found as $(t,x) \mapsto p(t,x) = q(\Phi(t),x)$, where $q$ uniquely solves the time-changed PDE for the corresponding backward-delay function $\eta[\Phi]$.
From there, our general aim is to show that this time-changed characterization is preserved in the presence of blowups, thereby justifying dPMF dynamics as a convenient modeling framework to  analytically study mean-field dynamics with blowups.


\subsection{Results} 
Our main result is to characterize explosive dPMF dynamics via a fixed-point problem bearing on a regularized time-changed dynamics.
To state this fixed-point problem, we first need to define a notion of initial conditions in the time-changed picture. 
In principle, the most general initial conditions for the original dPMF dynamics are specified by two measures $(p_0,f_0)$ in $\mathcal{M}(\mathbbm{R}^+) \times \mathcal{M}([-\epsilon,0))$, where 
$\mathcal{M}(I)$ denotes the space of nonnegative measures over the interval $I \subset \mathbbm{R}$.
Moreover, to represent a probability measure, $(p_0,f_0)$ must also satisfy the normalization condition:
\begin{eqnarray}
\int_0^\infty p_0(x) \, \dd x + \int_{-\epsilon}^0 f_0(t) \, \dd t =1 \, .
\nonumber
\end{eqnarray}
With this in mind, the time-changed version of the above initial conditions is specified as follows:

\begin{definition}\label{def:initCond_int}
Given normalized initial conditions $(p_0,f_0)$ in $\mathcal{M}(\mathbbm{R}^+) \times \mathcal{M}([-\epsilon,0))$, the initial conditions for the time-changed problem are defined by $(q_0,g_0)$ in $\mathcal{M}(\mathbbm{R}^+) \times \mathcal{M}([\xi_0,0))$ such that
\begin{eqnarray}
q_0=p_0 \quad \mathrm{and} \quad g_0 = \frac{\dd G_0}{\dd \sigma} \quad \mathrm{with} \quad G_0=(\mathrm{id} - \nu \Psi_0)/\lambda\, ,
\nonumber
\end{eqnarray}
where the function $\Psi_0$ and the number $\xi_0$ are given by:
\begin{eqnarray}
\Psi_0(\sigma) 
&=& 
\inf  \left\{  t \geq 0 \, \bigg \vert \, \nu t + \lambda \int_0^t f_0(s) \, \mathrm{d}s > \sigma \right\} \, , \nonumber\\
\xi_0&=& -\nu \epsilon -\lambda \int_{-\epsilon}^0 f_0(t) \, \dd t < 0 \, .
\nonumber
\end{eqnarray}
\end{definition}

Equipped with the above notion of initial conditions, we are in a position to state the fixed-point problem characterizing possibly explosive dPMF dynamics.
This fixed-point problem will be most conveniently formulated in term of the inverse time change $\Psi=\Phi^{-1}$, which can generically be assumed to be a continuous, nondecreasing function.
Given an inverse time change $\Psi$, the time change $\Phi$ can be recovered as the right-continuous inverse of $\Psi$.
In the time-changed picture, blowups happen if the inverse time change $\Psi$ becomes locally flat and a synchronous event happens if $\Psi$ remains flat for a finite amount of time.
Informally, flat sections of $\Psi$ unfold blowups by freezing time in the original coordinate $t$, while allowing time to pass in the time-changed coordinate $\sigma$.
Such unfolding of blowups in the time-changed picture will allow for the  following characterization of inverse time change $\Psi$, which remains valid for explosive dPMF dynamics.

\begin{theorem}\label{th:fixedpoint_int}
Given time-changed initial conditions $(q_0, g_0)$ in $\mathcal{M}(\mathbbm{R}^+) \times \mathcal{M}([\xi_0,0))$, the  inverse time change $\Psi$ satisfies the fixed-point problem
\begin{eqnarray}
\forall \; \sigma \geq 0 \, , \quad \Psi(\sigma) = 
\left\{
\begin{array}{ccc}
\left( \sigma - \lambda \int_0^\sigma g_0(\xi) \dd \xi \right) / \nu   							& \quad \mathrm{if} &  -\xi_0 \leq \sigma < 0 \, , \vspace{5pt} \nonumber\\
\sup_{0 \leq \xi \leq \sigma} \big( \xi - \lambda G[\eta](\xi)\big) / \nu   & \quad  \mathrm{if}  &  \sigma \geq 0  \, .
\end{array}
\right.
\end{eqnarray}
where $G[\eta]$ is the smooth cumulative flux associated to a linear diffusion dynamics with time-inhomogeneous backward-delay function $\eta: \mathbbm{R}^+ \to \mathbbm{R}^+$.
Given a backward-delay function $\eta$, the cumulative flux $G$ is given as the unique solution to the quasi-renewal equation
\begin{eqnarray}\label{eq:DuHamel_int}
G(\sigma) = \int_0^\infty H(\sigma,x) q_0(x) \, \dd x  + \int_0^\sigma H(\sigma-\tau,\Lambda) \, \dd G(\tau-\eta(\tau))   \, ,
\end{eqnarray}
where by convention we set $G(\sigma)=G_0(\sigma)=\int_0^\sigma g_0(\xi) \dd \xi$ if  $-\xi_0 \leq \sigma < 0$.
Moreover, the integration kernel featured in \eqref{eq:DuHamel_int} is specified as $\sigma \mapsto H(\sigma,x) = \Prob{\tau_x \leq \sigma} $, where $\tau_x$ is the first-passage time to zero of a Wiener process started in $x>0$ and with negative unit drift.  
Finally, the fixed-point nature of the problem follows from the definition of the backward-delay function $\eta$ as the $\Psi$-dependent time-wrapped version of the constant delay $\epsilon$:
\begin{eqnarray}
\eta(\sigma) = \sigma - \Phi(\Psi(\sigma) -\epsilon)   \quad \mathrm{with} \quad \Phi(t)= \inf  \left\{  \sigma \geq \xi_0 \, \big \vert \,\Psi(\sigma) > t \right\}  \, , \nonumber
\end{eqnarray}
for which we consistently have $\eta(0)=- \Phi(-\epsilon)=\xi_0$.
\end{theorem}

We will show that the time-changed formulation of dPMF dynamics yields the existence and uniqueness of dPMF dynamics under an additional assumption about the initial conditions.
That assumption bears on the distribution of active processes at starting time and states that $q_0$ is a locally smooth near zero with $q_0(0)=0$ and $\partial_x q_0(0) /2<1/\lambda$.
In view of \eqref{eq:cPMF}, such an additional assumption precludes a blowup from happening instantaneously.
In turn, this will allow us to show the following existence and uniqueness result:

\begin{theorem}
Given time-changed initial conditions $(q_0,g_0)$ in $\mathcal{M}(\mathbbm{R}^+) \times \mathcal{M}([\xi_0,0))$ such that  $q_0(0)=0$ and $\partial_x q_0(0) /2<1/\lambda$, the fixed-point problem defined in Theorem \ref{th:fixedpoint_int} admits a unique smooth solution $\Psi$ on $[0,S_1]$, where
\begin{eqnarray}
S_1=\inf \lbrace  \sigma>0 \, \vert \, \Psi'(\sigma) \leq 0 \rbrace > 0 \, .
\nonumber
\end{eqnarray}
Moreover, if $S_1<\infty$ and $\Psi''(S_1^-)<0$, this solution can be uniquely continued on $[S_1,S_1+\lambda \pi_1)$ as a constant, where $\pi_1$ solves the self-consistent equation:
\begin{eqnarray}\label{eq:defjump_int}
\pi_1 = \inf \left\{ p \geq 0 \, \bigg \vert \, p > \int_{0}^\infty H(\lambda p, x )  q(S_1,x ) \, \dd x \right\} \, .
\end{eqnarray}
\end{theorem}

The above result indicates how the time-changed process $Y_\sigma$ resolves a blowup episode by alternating two types of dynamics.
Before blowups, the dynamics $Y_\sigma$ is that of an absorbed linear diffusion with resets. 
These resets occur with time-inhomogeneous delays, which depends on the inverse time change $\Psi$.
At the blowup onset $S_1$, $\Psi$ becomes locally flat, indicating that the original time $t=\Psi(\sigma)$ freezes, thereby stalling resets.
As a result, after the blowup onset in $S_1$, the dynamics  of $Y_\sigma$ remains that of an absorbed linear diffusion but without resets.
Such a dynamics persists until a self-consistent blowup exit condition is met in $S_1+\lambda \pi_1$.
This condition follows from \eqref{eq:defjump_int}  and states that it must take $\lambda \pi_1$  time-changed units for a fraction $\pi_1$ of processes to inactivate during a blowup episode.
Finally, note that the generic condition that $\lim_{\sigma \to S_1} \partial^2_\sigma \Psi(\sigma)<0$, which we refer to the full-blowup condition, implies that the blowup is marked for the original dynamics $X_t=Y_{\Phi(t)}$ at $T_1=\Psi(S_1)$ by a well-characterized rate divergence: $f(t) \sim_{t \to T_1^-} 1/\sqrt{T_1-t}$.

In principle, dPMF dynamics could be continued past a blowup episode, and possibly even extended to the whole half-line $\mathbbm{R}^+$.
Showing this in our framework would require to check that $(1)$ the so-called  $(i)$ nonexplosive exit conditions $\partial_\sigma q(S_1+\lambda \pi_1,0)/2<1/\lambda$ and $(ii)$ full-blowup condition $\partial^2_\sigma \Psi(S_1^-)<0$ are constitutively satisfied and that $(2)$ blowup times do not have an accumulation point.
This program is beyond the scope of this work and is the topic of another manuscript \cite{TTLS}, where we show that these conditions hold for large enough interaction parameter $\lambda$.
Here, we only state the main result  of \cite{TTLS} about the existence of global explosive dPMF dynamics:

\begin{theorem}
For large enough $\lambda>\Lambda$, there exists explosive dPMF dynamics defined over the whole half-line $\mathbbm{R}^+$, with a countable infinity of blowups.
These blowups occurs at consecutive times $T_k$  with size $\pi_k$, $k \in \mathbbm{N}$, and are such that $\pi_k$ and $T_{k+1}-T_k$ are both bounded away from zero.
\end{theorem}


\subsection{Methodology} 
Overall, the main interest of our approach lies in our ability to resolve a singular dynamics by mapping it onto a regular dynamics, but via possibly discontinuous change of time.
Such an approach avoids resorting to convergence arguments in sample-path spaces equipped with the Skorokhod topology. Our approach proceeds in four steps:

First, we infer the interacting-particle systems approximating the dPMF dynamics by modifying the systems known to approximate cMF dynamics \cite{Delarue:2015,Delarue:2015b}.
This involves considering noisy synaptic interactions, whereby spiking updates in downstream neurons are i.i.d. following a normal law with mean and variance equal to $w_N=\lambda/N$, where $N$ denotes the number of neurons.
Conjecturing propagation of chaos \cite{Sznitman:1989} in the infinite size limit $N \to \infty$ allows us to justify the form of the PDE problem associated to dPMF dynamics, which is only well-posed for nonexplosive dynamics.
In order to extend this PDE characterization  to explosive dPMF dynamics, we must give a weak formulation to the associated PDE problem.
Due to the mean-field nature of dPMF dynamics, this weak formulation involves considering the cumulative flux $F$ as an auxiliary unknown  function.

Second, we define the linear, time-inhomogeneous PDE problem associated to the process $Y_\sigma$ obtained by time change of nonexplosive dPMF dynamics: $X_t=Y_{\Phi(t)}$.
The hypothesis of nonexplosive dynamics is necessary to ensure that the time change $\Phi$ introduced in \eqref{eq:sigmaTime_int} is smooth.
However, considerations from renewal analysis show that the obtained time-changed PDE problem is actually unconditionally well-posed, independent on the assumption of smoothness of $\Phi$.
By this, we mean that for any choice of nondecreasing function $\Phi$, the latter PDE problem has well-behaved solutions $(\sigma,x) \mapsto q[\Phi](\sigma,x)$, in the sense that by contrast with $F=G \circ \Phi$, the associated cumulative flux $G$ is always a smooth function of time.
Therefore, this time-changed picture provides us with a natural framework to define a notion of explosive dPMF dynamics, with possibly many blowups.

Third,  we show that among solutions parametrized by increasing functions $\Phi$, candidate solutions of the form $(t,x) \mapsto p(t,x)=q[\Phi](\Phi(t),x)$ are also weak solutions for dPMF dynamics if and only if the time change $\Phi$ satisfies the self-consistent equation \eqref{eq:selfcons_int}.
Technically, showing this point relies on the substitution formula for nonsmooth changes of variable \cite{Falkner:2012} as well as on the Vol'pert superposition principle \cite{Volpert:1967,DalMaso:1991}.
This result justifies reducing the analysis of dPMF dynamics to the study of a delayed, nonlinear, integral equation defining the fixed-point problem of Theorem \eqref{th:fixedpoint_int}.
Crucially, for bearing on possibly discontinuous increasing time change $\Phi$, this approach fully captures explosive dPMF dynamics.

Fourth, we show that the fixed-point problem of Theorem \eqref{th:fixedpoint_int} admits local solutions for a class of initial conditions that exclude instantaneous blowup.
Under such initial conditions, we establish the existence of initial smooth dPMF dynamics via a contraction argument.
This contraction argument relies on the fact that for nonzero refractory period $\epsilon>0$, the quasi-renewal equation \eqref{eq:DuHamel_int} loses its renewal character at small enough timescale.
Then, repeated application of the Banach fixed-point theorem allows one to specify a smooth solution $\Phi$ up to the first putative blowup time $T_1$.
In the time-changed coordinate $\sigma$, the onset of a blowup episode at $S_1=\Phi(T_1)$ corresponds to withholding resets, which leads to a natural self-consistent equation for blowup sizes $\pi_1$.
Such blowups are certain in the large interaction regime $\lambda \geq \Lambda$ and can be shown to have physical size, in the sense that they must correspond to a finite fraction $0< \pi_1<1$.




\subsection{Structure}

In Section \ref{sec:modelDef}, we introduce the delayed Poissonian (dPMF) dynamics as the conjectured mean-field limit of a certain interacting-particle systems and define its associated weak PDE formulation.
In Section \ref{sec:timeChange}, we  show that in the absence of blowups, dPMF dynamics are fully determined by a time change that maps the original time-homogeneous nonlinear dynamics on time-inhomogeneous linear dynamics. 
In Section \ref{sec:fixedPoint}, we exploit the weak formulation of dPMF dynamics to show that the proposed time-changed formulation remains valid in the presence of blowups, exhibiting  the fixed-point problem that characterizes admissible time changes.
In Section \ref{sec:local}, we show that the fixed-point problem admits local solutions with blowups and resolve analytically these blowups.


\section{The delayed Poisson-McKean-Vlasov dynamics}\label{sec:modelDef}

In this section, we justify the consideration of dPMF dynamics to study the emergence and persistence of blowups in mean-field neural models.
Conjecturing that propagation of chaos holds in the infinite-size limit, we justify the McKean-Vlasov equations defining dPMF dynamics in the absence of blowups.
We then leverage these equations to elaborate a weak formulation that allows for the consideration of explosive dPMF dynamics.


\subsection{Finite-size stochastic model}

We start by defining the finite-size version of the dPMF dynamics in terms of a particle system whose dynamics is unconditionally well-posed.
This particle system consists of a network of $N$ interacting processes $X_{N,i,t}$, $1\leq i \leq N$, whose interaction dynamics is as follows:
$(i)$ Whenever a process $X_{N,i,t}$ hits the spiking boundary at zero, it instantaneously enters an inactive refractory state.
$(ii)$ At the same time, all the other active processes $X_{N,j,t}$ (which are not in the inactive refractory state) are respectively decreased by amounts $w_{N,j,t}$, which are independently drawn from a normal law with mean and variance equal to $\lambda/N$.
$(iii)$ After an inactive (refractory) period of duration $\epsilon>0$, the process $X_{N,i,t}$ restarts its autonomous stochastic dynamics from the reset state $\Lambda>0$.
$(iv)$ In between spiking/interaction times, the autonomous dynamics of active processes follow independent drifted Wiener processes with negative drift $-\nu$.
Correspondingly, an initial condition for the network is given by specifying the starting values of the active processes, i.e. $X_{N,i,0}>0$ if $i$ is active, and the last hitting time of the inactive processes,  i.e., $-\epsilon  \leq \rho_{N,i,0} \leq 0$, if $i$ is inactive.

The above dynamics can be conveniently recapitulated in terms of the stochastic differential equations governing the interacting processes $X_{N,i,t}$, $1\leq i \leq N$.
For $1\leq i \leq N$, these equations takes the general form
\begin{eqnarray}\label{eq:stochEq}
X_{N,i,t} = X_{N,i,0} - \int_0^t \mathbbm{1}_{ \{ X_{N,i,s^-} >0 \} } \dd Z_{N,i,s} + \Lambda M_{N,i,t-\epsilon} \, , 
\end{eqnarray}
where $Z_{N,i,t}$, $1\leq i \leq N$, denote continuous-time driving processes with Poisson-like attributes and 
where $M_{N,i,t}$, $1\leq i \leq N$, are increasing processes counting the number of times that $X_{N,i,t}$ hits the threshold $\Lambda$ before $t$.
The driving processes $Z_{N,i,t}$ are specified in term of cumulative drift functions  $\Phi_{N,i}$ according to
\begin{eqnarray}
Z_{N,i,t} = \Phi_{N,i}(t) +  W_{i,\Phi_{N,i}(t)} \, ,
\nonumber
\end{eqnarray}
where $W_{i,t}$, $1\leq i \leq N$, are independent Wiener processes.
Thus-defined, the driving processes $Z_{N,i,t}$ exhibit Poisson-like attributes in the sense that they have constitutive unit Fano factor.
By contrast, classical particle-system approaches consider driving processes with variable Fano Factor of the form $\Phi_{N,i}(t) +  W_{i,t}$.
The assumption of a constant Fano factor is the key to making an analytical treatment of blowups possible in the mean-field limit.

As generic cumulative functions,  the functions $\Phi_{N,i}$ are only assumed to be right-continuous with left limits, which we refer to as being \emph{c\`{a}dl\`{a}g} following classical probabilistic conventions.
These \emph{c\`{a}dl\`{a}g} cumulative functions $\Phi_{N,i}$ are naturally defined in terms of the counting processes $M_{N,j,t}$ as
\begin{eqnarray}
\Phi_{N,i}(t) = \nu t + \frac{\lambda}{N} \sum_{j\neq i} M_{N,j,t} \nonumber \, ,
\end{eqnarray}
showing that the jump discontinuities in $\Phi_{N,i}$ model interneuronal interactions.
In turn, the counting processes $M_{N,j,t}$ are defined as
\begin{eqnarray}
M_{N,j,t} = \sum_{n>0} \mathbbm{1}_{[0,t]}(\rho_{N,j,n}) \, , \nonumber 
\end{eqnarray}
where $\rho_{N,i,n}$, $n \geq 0$, denote the successive first-passage times of $X_{N,i,t}$ to the zero spiking threshold.
These times are formally defined for all $n\geq 0$ by 
\begin{eqnarray}
\rho_{N,i,n+1} = \inf \left\{  t>r_{N,i,n} \, \big  \vert \,  X_{N,i,t}  \leq 0 \right\} \, . \nonumber 
\end{eqnarray}
where by convention, we set $r_{N,i,0} = 0$ for processes that are active in zero with $X_{N,i,0}>0$ and where $r_{N,i,n} = \rho_{N,i,n}+\epsilon$, $n\geq 1$, denote the successive delayed reset times.

Thus defined, the particle-system dynamics is self exciting: 
every spiking event of a neuron $i$  hastens the spiking of other neurons $j \neq i$ by bringing their states closer to the zero threshold boundary.
Moreover, the particle-system dynamics allows for synchronous spiking as whenever neuron $i$ spikes due to its autonomous dynamics, we generically have that $\Prob{X_{N,j,t} \in (0,w_{N,j,t}]}>0$. 
If the process $X_{N,i,t}$ first hits zero at time $\rho$,  \eqref{eq:stochEq} implies that $X_{N,i,t}$ remains in zero for all $t$ in $(\rho, \rho+\epsilon)$, until it receives an instantaneous kick that enforces a reset in $\Lambda$ at time $r=\rho+\epsilon$.
Thus, \eqref{eq:stochEq} formally identifies the refractory state with zero. 
However, it will prove more convenient to consider the inactive state as an isolated inactive state away from zero.
The reason for this is that such a consideration avoid modeling inactive processes via Dirac-delta mass in zero, so that regular absorbing boundary conditions in zero can be enforced.

Mathematically, the  benefit of including an inactive period $\epsilon$ is to ensure the uniqueness of the particle-system dynamics during spiking avalanches, thereby ensuring that the overall dynamics is well-posed. 
Spiking avalanches occurs when the spiking of a neuron triggers the instantaneous spiking of other neurons.
Neurons that engage in a spiking avalanche can be sorted out according to a generation number.
Generation $0$ contains the lone triggering neuron which is driven to the absorbing boundary by its autonomous dynamics.
Generation $1$ comprises all those neurons that spike due to interactions with the triggering neuron alone.
In general, generation $k>1$, comprises all the neurons that spike from interacting with the neurons of the previous generations alone.
In the absence of a post-spiking inactive period ($\epsilon=0$), it is ambiguous whether the neurons from previous generations are impacted by the spiking of neurons from the following generations.
However, in the presence of an inactive period ($\epsilon>0$), neurons from previous generations are unresponsive to neurons from following generations due to post-spiking transient inactivation.
Accordingly, as a variation on \cite{Delarue:2015b}, we resolve the ambiguity of spiking avalanche in the absence of inactive period by only considering the so-called ``physical dynamics'', obtained from delayed dynamics in the limit $\epsilon \to 0^+$.
These ``physical dynamics'' assume that independent of their generation number, every neuron engaging in a spiking avalanche at time $t$ spikes only once and resets in $\Lambda$ at $t^+$.
We conclude by noting that the delayed dynamics introduced here differ from those considered in \cite{Delarue:2015b}, where the delay bears on the interactions rather than the resets.
This distinction is important as by contrast with reset-delayed dynamics, interaction-delayed dynamics are not prone to explosions.

\subsection{Mean-field dynamics under propagation of chaos}

The particle-system dynamics introduced above primarily differs from the classically considered one by its Poisson-like attributes.
In classically defined particle systems, the jump discontinuities of the driving inputs have fixed size $\lambda/N$ instead of being i.i.d according to a normal law with mean and variance equal to $\lambda/N$.
In \cite{Delarue:2015b}, Delarue {\it et al.} show that the property of propagation of chaos holds in the infinite-size limit of classically defined particle systems.
This property establishes that in the infinite-size limit, a representative process $X_t=\lim_{N \to \infty} X_{N,i,t}$ follows a mean-field dynamics satisfying the PDE problem \eqref{eq:classicalMV} and  \eqref{eq:classicalFlux} originally introduced in \cite{Caceres:2011,Carrillo:2013}.
This particle-system-based approach automatically yields the existence of---possibly explosive---solutions to the PDE problem  \eqref{eq:classicalMV} and \eqref{eq:classicalFlux}.
Here, by contrast with \cite{Delarue:2015b} and in line with \cite{Caceres:2011,Carrillo:2013}, we only conjecture propagation of chaos to motivate the form of the PDE problem defining a novel mean-field dynamics that is prone to blowup.
We then consider these dynamics on their own merit, independent of the conjecture of propagation of chaos. 

The propagation of chaos states that for exchangeable initial conditions, the processes $X_{N,i,t}$, $1 \leq i \leq N$, become i.i.d. in the limit of infinite-size networks $N \to \infty$, so that each individual process follows a mean-field dynamics.
We refer to such a mean-field dynamics as a cMF dynamics for the classical model and a dPMF dynamics for the Poisson-like model.
For both models, the mean-field interaction governing the dynamics of a representative process $X_t$ is mediated by a deterministic cumulative drift $\Phi =\lim_{N \to \infty} \Phi_{N,i}$. 
Formally, this deterministic drift is defined as $\Phi(t) = \nu t + \lambda \Exp{ M_t}$, where the process $M_t$ counts the successive first-passage times of the representative process $X_t$ to the zero spiking threshold:
\begin{eqnarray}\label{eq:MFM}
M_t = \sum_{n>0} \mathbbm{1}_{[0,t]}(\rho_n) \, , \quad \mathrm{with} \quad \rho_{n+1} = \inf \left\{  t>r_n = \rho_n+\epsilon  \, \big  \vert \,  X_{t}  \leq 0 \right\} \, .
\end{eqnarray}
In the following, we will denote the increasing function $\Exp{ M_t}$ by $F(t)$.
In the context of the associated PDE problem, we will refer to $F$  as the cumulative flux function through the zero absorbing boundary. 
Observe that by definition, the function $F$ is an increasing \emph{c\`adl\`ag} function.
This allows one to define the instantaneous firing rate $f$ in the distribution sense as the Radon-Nikodym derivative of $F$ with respect to the Lebesgue measure
$f = \dd F /\dd t$.
Correspondingly, synchronous events whereby a finite fraction of processes spike simultaneously are marked by Dirac-delta mass in $f$.

By contrast with cMF models, the deterministic cumulative drift $\Phi(t)=\nu t +\lambda F(t)$ constitutively impacts neurons with Poissonian attributes in dPMF models, i.e., via a process $Z_t = \Phi(t)+W_{\Phi(t)}$, where $W$ is a driving Wiener process.
Accordingly, the stochastic dPMF dynamics of a representative process is given by
\begin{eqnarray}\label{eq:stochEqMF}
X_t = X_0 - \int_0^t \mathbbm{1}_{ \{ X_{s^-} >0 \} } \dd Z_s + \Lambda M_{t-\epsilon} \, .
\end{eqnarray}
The above equation fully defines dPMF dynamics. 
Because of the self-interaction terms, dPMF dynamics are prone to blowups for large enough interaction coupling and/or for initial conditions that are concentrated near the boundary.
Actually, just as for cMF models, we will see that the cumulative drift $\Phi$ can exhibit $(i)$ singular blowups, corresponding to a divergence of the reset rate $f$ in finite time and $(ii)$ jump discontinuities whereby a finite fraction of the processes spike at the same time.
Our goal is to characterize analytically the emergence of these blowups.
This will require first defining the PDE problem associated to dPMF dynamics in the absence of blowups.

\subsection{McKean-Vlasov equations under smoothness assumptions}\label{subsec:MKV}

For weak interaction, i.e., $\lambda<\Lambda$, we expect dPMF dynamics to be nonexplosive for initial conditions far enough from the spiking threshold, e.g., $p(0,x)=\delta_{x_0}(x)$ with sufficiently large $x_0>0$.
This motivates defining the PDE problem associated to dPMF dynamics under strong regularity assumptions.
Specifically, let us assume that $t \mapsto F(t) = \Exp{M_t}$ is smooth on $[0,T)$ for some $T>0$. 
Then, $f(t) = F'(t)$ represents the nonnegative, smooth, mean-field rate of inactivation.
Under such regularity assumptions, a representative process $X_t$ satisfying \eqref{eq:stochEqMF} admits a probability density $(t,x) \mapsto p(t,x)$ which solves the Fokker-Plank equation
\begin{eqnarray}\label{eq:FP}
\partial_t p = (\nu+\lambda f(t)) \left( \partial_x p + \partial^2_{x} p/2 \right) + \mathbbm{1}_{\{ t>\epsilon \} }f(t-\epsilon) \delta_{\Lambda} \, , 
\end{eqnarray}
with absorbing boundary condition in  $p(t,0) = 0$.
The latter absorbing condition ensures that the process becomes inactive upon reaching zero.
The Dirac-delta source term models the reset in $\Lambda$ of newly activated processes, which happens in $t$ with delayed rate $f(t-\epsilon)$.

To be consistent, the mean-field dynamics specified by \eqref{eq:FP}  needs to conserve the total probability.
This conservation requirement implies that
\begin{eqnarray}\label{eq:cons1}
\partial_t \left( \int_{0}^\infty \! p(t,x) \, \dd x \right) =\mathbbm{1}_{\{t>\epsilon \}}f(t-\epsilon) - f(t) \, .
\end{eqnarray}
Using \eqref{eq:FP}, we can evaluate the left term above as
\begin{eqnarray}
\partial_t \left( \int_{0}^\infty \! p(t,x) \, \dd x \right)  &=&  \int_{0}^\infty  \big( \nu  +  \lambda f(t) \big) \left( \partial_x p(t,x) + \partial^2_{x} p(t,x) /2 \right) \, \dd x  \nonumber \\
 &&  +  \int_{0}^\infty \mathbbm{1}_{\{t>\epsilon \}}f(t-\epsilon) \delta_{\Lambda}(x) \, \dd x \, .  \nonumber 
\end{eqnarray}
Performing integration by parts with absorbing boundary condition in zero then yields
\begin{eqnarray}
\partial_t \left( \int_{0}^\infty \! p(t,x) \, \dd x \right) = -\big(\nu \! + \! \lambda f(t)\big) \partial_x p(0,t)/2 + f(t-\epsilon) \, ,  \nonumber 
\end{eqnarray}
which together with \eqref{eq:cons1} imposes the self-consistent conservation condition $f(t) = (\nu \! + \! \lambda f(t)) \partial p_x(t,0) /2$, ultimately yielding:
\begin{eqnarray}
f(t) = \frac{ \nu \partial_x p(0,t)}{2 -  \lambda \partial_x p(0,t)} \, .  \nonumber 
\end{eqnarray}
The Fokker-Planck equation \eqref{eq:FP} and the above conservation condition fully specify the mean-field dynamics.
As the coefficients of the Fokker-Plank equation depends on its solution via a boundary flux term, the mean-field dynamics is actually a nonlinear Markov evolution of the McKean-Vlasov type.


Finally, observe that to avoid blowups, we have only considered initial conditions of the form $p_0(x)=\delta_{x_0}(x)$, with initially empty inactive state.
However, the PDE problem defined above can be considered for more generic initial conditions, at the possible cost of not having any regular solutions.
According to the delayed nature of the dynamics, these generic initial conditions are naturally specified by
\begin{eqnarray}
p(0,x) =  p_0(x) \quad \mathrm{and} \quad  f(t) = f_0(t) \, ,   \quad  \epsilon  \leq t< 0 \, .  \nonumber 
\end{eqnarray}
with $(p_0,f_0)$ in $\mathcal{M}(\mathbbm{R}^+) \times \mathcal{M}([-\epsilon,0))$, 
where $(p_0,f_0)$ satisfies the normalization condition
\begin{eqnarray}\label{eq:normCond}
 \int_{-\epsilon}^0 f_0(s) \, \dd s  = 1-\int_{0}^\infty p_0(x) \, \dd x \, .
\end{eqnarray}
Given this notion of initial conditions, we define the McKean-Vlasov PDE problem associated to smooth dPMF dynamics as:

\begin{definition}\label{def:delayedPDE} 
Given normalized initial conditions $(p_0,f_0)$ in $\mathcal{M}(\mathbbm{R}^+) \times \mathcal{M}([-\epsilon,0))$, the PDE problem associated to a dPMF dynamics consists in finding  the density function $(t,x) \mapsto p(t,x)= \dd \Prob{0<X_t<x} / \dd x$ solving 
\begin{eqnarray}\label{eq:delayedPDE}
\partial_t p  &=& \big(\nu + \lambda f(t) \big) \left( \partial_x p + \partial^2_{x} p /2 \right) + f(t-\epsilon) \delta_{\Lambda}  \, ,
\end{eqnarray}
on $[0,T) \times [0, \infty)$ for some (possibly infinite) $T>0$ and with absorbing and conservation conditions given by
\begin{eqnarray}\label{eq:delayedPDEcond}
p(t,0) = 0 \quad \mathrm{and} \quad
f(t)   =  \frac{\nu \partial_x p(t,0)}{2  - \lambda \partial_x p(t,0)} \, .
\end{eqnarray}
\end{definition}

The challenge posed by the emergence of blowups is to make sense of  \eqref{eq:delayedPDE} for instantaneous flux $f$ exhibiting finite-time divergence, possibly followed by Dirac-delta mass, corresponding to a jump discontinuity in $F$.
Such singularities present themselves whenever an initially smooth dynamics is such that $t \mapsto \partial_x p(t,0)/2$ reaches $1/\lambda$ in finite time.
As intuition suggests, this blowup criterion will generally be met for sufficiently large interaction parameter.
In particular, we will see that given initial conditions of the form $p_0=\delta_{x_0}$ with $x_0>0$, there is a constant $C_{x_0}$ which only depends on $x_0$ such that a blowup occurs in finite time for all $\lambda>C_{x_0}$.

\subsection{Weak formulation for explosive dPMF dynamics}

In this section, our goal is to propose a weak formulation of the PDE problem \ref{def:delayedPDE} that is amenable to capture explosive dPMF dynamics.
This formulation bears on candidate density functions $(t,x) \mapsto p(t,x)$ in the space of distributions defined as the dual of  $C_{0,\infty}([-\epsilon,T) \times \mathbbm{R})$, the set of compactly supported, smooth functions $u:[-\epsilon,T)\times \mathbbm{R} \rightarrow \mathbbm{R}$.
We derive the announced weak formulation by first considering nonexplosive solutions of the PDE problem \ref{def:delayedPDE}.
For all nonexplosive solutions $(t,x) \mapsto p(t,x)$ and all test functions $u$ in $C_{0,\infty}([-\epsilon,T) \times \mathbbm{R})$, we must have
\begin{eqnarray*}
0&=&\int_0^T \int_{0}^\infty   \left[ \big(\nu + \lambda f(t) \big) \mathcal{L}[p](t,x)  + f(t-\epsilon) \delta_{\Lambda} - \partial_t p(t,x) \right] u(t,x) \, \dd t\, \dd x \,, \\
&=&\int_0^T     \int_{0}^\infty \big[ \big(\nu + \lambda f(t) \big) \mathcal{L}[p](t,x) - \partial_t p(t,x) \big] u(t,x) \, \dd x \, \dd t \\
&& \hspace{200pt} +  \int_0^T f(t-\epsilon) u(t,\Lambda) \, \dd t \, , 
\end{eqnarray*}
where $\mathcal{L}$  denotes the operator $\mathcal{L}=\partial_x  + 1/2\partial^2_{x} $ for brevity.
Taking into account the absorbing boundary condition, integration by parts with respect to space yields
\begin{eqnarray*}
 \int_{0}^\infty  \mathcal{L}[p](t,x) u(t,x) \, \dd x =  \int_{0}^\infty  p(t,x) \mathcal{L}^\dagger [u](t,x)\, \dd x +  \frac{1}{2}\partial_x p(t,0) u(t,0)  \, , 
\end{eqnarray*}
so that integration by part with respect to time produces
\begin{eqnarray*}
\int_{0}^\infty \big[ p(T,x)u(T,x)-p(0,x)u(0,x) \big] \, \dd x &=& \\
&& \hspace{-120pt}\int_0^T\int_{0}^\infty  p(t,x) \big[  \big(\nu + \lambda f(t) \big) \mathcal{L}^\dagger [u](t,x) + \partial_t u(t,x) \big] \, \dd x \, \dd t  \nonumber\\
&& \hspace{-120pt} + \int_0^T \big[ f(t-\epsilon)u(t,\Lambda)-\big(\nu + \lambda f(t) \big) \partial_x p(t,0) u(t,0)/2 \big] \, \dd t \, .
\end{eqnarray*}
Remembering the flux conservation condition $\partial_x p(t,0)=f(t)/(\nu+\lambda f(t))$, we obtain the following weak characterization for nonexplosive solutions
\begin{eqnarray*}
\int_{0}^\infty \big[ p(T,x)u(T,x)-p(0,x)u(0,x) \big] \, \dd x &=& \\
&& \hspace{-120pt}\int_0^T\int_{0}^\infty  p(t,x) \big[  \big(\nu + \lambda f(t) \big) \mathcal{L}^\dagger [u](t,x) + \partial_t u(t,x) \big] \, \dd x \, \dd t  \nonumber\\
&& \hspace{-120pt} + \int_0^T \big[ f(t-\epsilon)u(t,\Lambda)- f(t) u(t,0) \big] \, \dd t \, .
\end{eqnarray*}
The above characterization involves the instantaneous flux $f$ as an unknown, which can be safely assumed to be a nonnegative integrable function.
With that in mind, one can see that the proposed characterization derived for nonexplosive solutions is well-posed for any candidate density function in the space of integrable distributions.
This leads to defining the notion of weak solution for the dPMF dynamics in the presence of blowups as follows:

\begin{definition}\label{def:weakPDE} 
Given normalized initial conditions $(p_0,f_0)$ in $\mathcal{M}(\mathbbm{R}^+) \times \mathcal{M}([-\epsilon,0))$, the density function $(t,x)\mapsto p(t,x)$ is a weak solution of the dPMF dynamics if and only if there is a bounded nondecreasing \emph{c\`adl\`ag} function $F:[-\epsilon,T) \to \mathbbm{R}$ with $\dd F/dt = f_0$ on $[-\epsilon,0)$  such that for all $u$ in $C_{0,\infty}([-\epsilon,T) \times \mathbbm{R})$, we have
\begin{eqnarray}\label{eq:weakPDE} 
\int_{0}^\infty \big[ p(T,x)u(T,x)-p(0,x)u(0,x) \big] \, \dd x 
&=& \\
&&\hspace{-140pt} \int_0^T\int_{0}^\infty  p(t,x) \big[\nu  \mathcal{L}^\dagger [u](t,x) + \partial_t u(t,x) \big] \, \dd x \,  \dd t \nonumber\\
&& \hspace{-140pt}+ \: \lambda  \int_0^T\int_{0}^\infty  p(t^-,x)  \mathcal{L}^\dagger [u](t,x)\, \dd x \, \dd F(t) \nonumber\\
&& \hspace{-140pt}+ \int_0^T u(t,\Lambda)  \, \dd F(t-\epsilon)  - \int_0^T   u(t,0) \, \dd F(t) \, .\nonumber\
\end{eqnarray}
\end{definition}

Clearly, all  nonexplosive solutions $(t,x)\mapsto p(t,x)$ of the PDE problem \ref{def:delayedPDE} are weak solutions of the PDE problem \ref{def:weakPDE} for $F$ equal to the cumulative flux integrating $f$ as defined in \eqref{eq:delayedPDEcond}.
It is also clear that by contrast with the PDE problem \ref{def:delayedPDE}, the definition of weak solutions allows for discontinuous function $F$.
In that respect, observe that we enforce that $F$ is \emph{c\`adl\`ag} to be consistent with the definition of the counting process $M_t$ given in \eqref{eq:MFM}.
When choosing $F$ to be \emph{c\`adl\`ag}, it is then necessary to specify the type of continuity of the integrand for integrals with respect to $F$ as the integrator to be well-defined.
In view of the predictable integrand in stochastic equation \eqref{eq:stochEqMF}, we consistently impose that the integrand be left-continuous whenever $F$ features as the integrator. 
Intuitively, we expect the function $F$ featuring in Definition \ref{def:weakPDE} to be uniquely related to a weak solution $(t,x) \mapsto p(t,x)$, just as for nonexplosive solutions.
This fact is established by the following proposition: 

\begin{proposition}\label{prop:uniqueF}
There is a unique nondecreasing \emph{c\`adl\`ag} function $F$ such that the density function $(t,x)\mapsto p(t,x)$ is a weak solution of the dPMF dynamics.
\end{proposition}

\begin{proof}
Observe that for integrable density functions, the defining property of weak solutions also holds for smooth test functions with bounded derivatives of all orders.
Then, specifying \eqref{eq:weakPDE} for $u=1$ yields
\begin{eqnarray}\label{eq:u1}
\int_{0}^\infty p(T,x)\, \dd x-\int_{0}^\infty p(0,x)  \, \dd x 
=
\big( F(T-\epsilon) -F (-\epsilon) \big)  - \big( F(T) -F (0) \big) \, .
\end{eqnarray}
Consider two functions $F_1$ and $F_2$ such that $p$ is a weak solution of the dPMF dynamics. The initial conditions impose that we have 
\begin{eqnarray*}
F_1(0)-F_1(-\epsilon)=F_2(0)-F_2(-\epsilon)=\int_{-\epsilon}^0 f_0(t) \, \dd t \, .
\end{eqnarray*}
Therefore, specifying \eqref{eq:u1} for $F_1$ and $F_2$ and forming the difference yields
\begin{eqnarray*}
F_1(t)-F_1(t-\epsilon)=F_2(t)-F_2(t-\epsilon) \, ,
\end{eqnarray*}
so that, $F_1-F_2$ is an $\epsilon$-periodic function.
As $F_1 = F_2$ on the interval $[-\epsilon,0)$, we necessarily have $F_1=F_2$  for all $t \leq 0$ by $\epsilon$-periodicity.
\end{proof}

In the following our strategy will be to use the above notion of weak solutions to screen candidate explosive solutions defined via time change for \emph{bona fide} dPMF dynamics.

\section{Linearization via implicitly defined time change}\label{sec:timeChange}

In this section, we show that under certain regularity assumptions, dPMF dynamics can be turned into noninteracting linear dynamics via a time change. 
We then interpret these dynamics probabilistically via renewal analysis to establish that they are constitutively well-posed, independent of the time-change function.
Such a realization provides the basis to define explosive dPMF dynamics in the time-changed picture.

\subsection{Conditionally linear dynamics}

Informally, dPMF dynamics admit blowups at those times $T$ for which the instantaneous inactivation flux $f$ diverges: $\lim_{t \to T^-} f(t) =\infty$.
Characterizing such blowup times analytically entails studying the PDE problem  \ref{def:delayedPDE} with the drift, diffusion, and reset coefficients that are all allowed to locally diverge.
In general, this is a hard problem that cannot be tackled analytically.
However, for drift and diffusion coefficients with Poisson-like attributes, the problem is tractable thanks to the availability of a regularized time-changed formulation.
Not surprisingly, the function $\Phi$ operating this time change can be guessed as the integral function of the drift:
\begin{eqnarray}\label{eq:Phi}
t \mapsto \Phi(t) = \nu t+ \lambda F(t) \, .
\end{eqnarray}
This approach suggests considering that solutions to the PDE problem \ref{def:delayedPDE} as parametrized by a time change $\Phi$, which shall be viewed as the fundamental unknown of the problem.
In light of \eqref{eq:Phi}, we shall look for solution time change $\Phi$ in the following class of functions: 

\begin{definition}\label{def:Phi}
We define the class of valid time changes $\mathcal{T}$ as the set of \emph{c\`adl\`ag} functions $\Phi:[-\epsilon, \infty)\to[\xi_0, \infty)$, such that their difference quotients are lower bounded by $\nu$: for all $y,x \leq -\epsilon$, $x \neq y$, we have
\begin{eqnarray*}
w_\Phi(y,x)=\frac{\Phi(y)-\Phi(x)}{y-x} \geq \nu \, .
\end{eqnarray*}
\end{definition}

In general, to be a valid time change, we only require a function $\Phi: \mathbbm{R}^+ \to  \mathbbm{R}^+$ to be a nondecreasing \emph{c\`{a}dl\`{a}g} function. 
This means that time changes $\Phi$ must exclude time-reversal point at which the changed time $\sigma = \Phi(t)$ would start flowing backward when the original time $t$ keeps moving forward.
Here, the time change Definition \ref{def:Phi} additionally imposes that $\Phi$ has no flat region as we have $\nu>0$.
As a result, specifying the inverse time change $\Phi^{-1}: \mathbbm{R}^+ \to  \mathbbm{R}^+$ as the right-continuous generalized inverse of $\Phi$, actually yields a continuous function $\Psi$.
It is then clear that $\Phi$ is the right-continuous inverse of $\Psi$.

\begin{definition}\label{def:Psi}
Given a time change $\Phi$ in $\mathcal{T}$, the inverse time change $\Phi^{-1}:[\xi_0,\infty) \to [-\epsilon,\infty)$ is defined as the continuous function 
\begin{eqnarray*}
\sigma \mapsto \Psi(\sigma)=\Phi^{-1}(\sigma) = \inf \left\{ t \geq 0 \, \vert \, \Phi(t) > \sigma \right\} \, .
\end{eqnarray*}
\end{definition}

Importantly, valid time changes include function $\Phi$ with discontinuous jumps---or equivalently flat regions for $\Psi$.
Such discontinuities will correspond to the occurrence of synchronous events, at those times for which inactivation on the absorbing boundary has finite probability. 
The key to unlocking these synchronous events is that the time change $\Phi$ maps the dynamics of an eventually singular, interacting dynamics onto that of a constitutively regular, noninteracting one.
When unfolding along the new time coordinate $\sigma=\Phi(t)$, this regular dynamics will only depend on $\Phi$ via the time wrapping of the refractory period $\epsilon$.
Such time wrapping is captured by the so-called backward-delay function, which is defined as follows: 

\begin{definition}\label{def:Eta}
Given the time change $\Phi$ in $\mathcal{T}$, we define the corresponding backward-delay function $\eta:[0,\infty) \to \mathbbm{R}^+$ by
\begin{eqnarray*}
\eta(\sigma) = \sigma - \Phi \left( \Psi(\sigma) - \epsilon \right) \, , \quad \sigma \geq 0 \, .
\end{eqnarray*}
We denote  the set of backward functions $\lbrace \eta[\Phi]\rbrace_{\Phi \in \mathcal{T}}$ by $\mathcal{W}$.
\end{definition}

As $w_\Phi \geq \nu$ for all $\Phi$ in $\mathcal{T}$, it is clear that for all $\eta$ in $\mathcal{W}$, we actually have $\eta  \geq \nu \epsilon$, so that all delays are bounded away from zero.
%
Time-wrapped-delay function $\eta$ in $\mathcal{W}$  will serve to parametrize the time-changed dynamics obtained via $\Phi$ in $\mathcal{T}$.
These time-changed dynamics will be that of a modified Wiener process $Y_\sigma$ with negative unit drift, inactivation on the zero boundary, and reset in $\Lambda$ after a refractory period specified by $\eta$.
Consequently, we define time-changed dynamics as the processes $Y_\sigma$ solutions to the following stochastic evolution:

\begin{definition}\label{def:Ysigma}
Denoting the canonical Wiener process by $W_\sigma$, we define the time-changed processes $Y_\sigma$ as solutions to the stochastic evolution
\begin{eqnarray}\label{eq:Ysigma}
Y_\sigma = -\sigma+ \int_0^\sigma \mathbbm{1}_{\{Y_{\xi^-}>0\}}\, \dd W_\xi + \Lambda N_{\sigma-\eta(\sigma)} \, ,\quad \mathrm{with} \quad N_\sigma = \sum_{n>0} \mathbbm{1}_{[\xi_0,t]}(\xi_n) \, ,
\end{eqnarray}
 where the process $N_\sigma$ counts the successive first-passage times $\xi_n$ of the process $Y_\sigma$ to the absorbing boundary:
\begin{eqnarray*}
\xi_{n+1} = \inf \left\{  \sigma >0  \, \big  \vert \, \sigma-\eta(\sigma) > \xi_n, Y_\sigma  \leq 0 \right\}  \, .
\end{eqnarray*}
\end{definition}

A time-changed process $Y_\sigma$ is uniquely specified by imposing elementary initial condition, which takes an alternative formulation: either the process is active $Y_0=x>0$ and $N_0=0$, either the process has entered refractory period at some earlier time $\xi$ so that $Y_0=0$ and $N_\sigma=\mathbbm{1}_{\sigma \geq \xi}$ for $\xi_0 \leq \sigma < 0$.
Generic initial conditions are given by considering that $(x,\xi)$ is sampled from some probability distribution on $\{ (0,\infty) \times \{ 0 \} \} \cup \{ \{ 0 \} \times [\xi_0,0) \}$.
This amounts to choosing a normalized pair of distributions $(q_0,g_0)$  in $\mathcal{M}(\mathbbm{R}^+) \times \mathcal{M}([\xi_0,0))$. 
The ensuing dynamics is well-defined as long as the backward-delay function $\eta \geq \nu \epsilon$ is locally bounded, which is always the case for valid time change $\Phi$.
Of particular interest is the fact that such dynamics can accommodate jump discontinuities in $\eta$.
This is perhaps best seen by considering the so-called backward-time function $\mathbbm{R}^+ \to [\xi_0,\infty )$, $\sigma \mapsto \sigma-\eta(\sigma)$, featured in the time-delayed counting process of \eqref{eq:Ysigma}.
When unambiguous, we will denote this backward-time function by $\xi$ and refer to it as the ``function $\xi$'' to differentiate from when $\xi$ plays the role of a real variable.
By construction,  the function $\xi$ satisfies $\xi(\sigma) =\sigma-\eta(\sigma)= \Phi \left( \Psi(\sigma) - \epsilon \right)$, and is thus a nondecreasing \emph{c\`adl\`ag} function, possibly admitting discontinuities and flat regions.
Specifically, denoting by $\mathcal{D}_\Phi$ the countable set of discontinuous time of $\Phi$ in $[-\epsilon,\infty)$, the function $\xi$ has discontinuities on the countable set
\begin{eqnarray*}
\mathcal{D}_\xi = \cup_{t \in \mathcal{D}_\Phi} \{ \inf \Psi^{-1}(\{ t+\epsilon \}) \} = \cup_{t \in \mathcal{D}_\Phi} \{ \inf \{\sigma \, \vert \, \Psi(\sigma)=t+\epsilon \}  \} \, ,
\end{eqnarray*}
whereas it is flat on the countable disjoint union of discontinuity intervals of $\Phi$:
\begin{eqnarray*}
\mathcal{J}_\xi = \mathcal{J}_\Psi = \cup_{t \in \mathcal{D}_\Phi} \{  \Psi^{-1}(\{ t \}) \}  = \cup_{t \in \mathcal{D}_\Phi} [\Phi(t^-), \Phi(t)] \, .
\end{eqnarray*}

The discontinuities and flat regions of the function $\xi$ will play the central part in explaining the occurrence of synchronous events in the original dPMF dynamics from analyzing the dynamics of $Y_\sigma$.
In this perspective, it is worth completing the time-changed picture of the dPMF dynamics by stating  the PDE problem attached to the dynamics of $Y_\sigma$:

\begin{definition}\label{def:qPDE}
Given a backward-delay function $\eta$ in $\mathcal{W}$ and some normalized initial conditions $(q_0,g_0)$ in $\mathcal{M}(\mathbbm{R}^+) \times \mathcal{M}([\xi_0,0))$, $\xi_0=-\eta(0)$,  the density function $(\sigma,x) \mapsto q(\sigma,x)=\dd \Prob{0<Y_\sigma \leq x}/\dd x$ solves the time-changed PDE problem 
\begin{eqnarray}\label{eq:qPDE}
\partial_\sigma q &=& \partial_x q +\frac{1}{2} \partial^2_{x} q  + \frac{\dd}{\dd \sigma} [G(\sigma-\eta(\sigma))] \delta_{\Lambda}  \, , 
\end{eqnarray}
with absorbing and conservation conditions respectively given by 
\begin{eqnarray}\label{eq:abscons}
q(\sigma,0)=0 \quad \mathrm{and} \quad g(\sigma )= \partial_\sigma G(\sigma)=\partial_x q(\sigma ,0) /2 \, .
\end{eqnarray}
\end{definition}

We refer to the time-changed PDE problem \ref{def:qPDE} as a regularized one because the emergence of  synchrony will only involves discontinuities in the backward-delay function rather than diverging drift and diffusion coefficients.
This is obvious from the fact that \eqref{eq:qPDE} features constant unit drift and diffusion coefficients, so that all the interactions present in the original problem will be mediated by $\eta$ in the time-changed dynamics.
The boundary condition \eqref{eq:abscons} identifies $g$ as the absorbing boundary flux of $Y_{\sigma}$, i.e., as its instantaneous inactivation rate.
Because of the nonhomogeneity of $\eta$, the reset rate with which $Y_{\sigma}$ activates is generally distinct from the  inactivation rate, as shown by the prefactor of the Dirac-delta source term in \eqref{eq:qPDE}.
Technically, this prefactor is defined as the Radon-Nikodym derivative of the measure specified by the cumulative function $\sigma \mapsto G(\sigma-\eta(\sigma))$ with respect to the Lebesgue measure on $[0,\infty)$.
This definition is justified by the fact that for all $\Phi$ in $\mathcal{T}$, $\sigma \mapsto \sigma-\eta(\sigma)=\Phi \left( \Psi(\sigma) - \epsilon \right)$ is a nondecreasing function and that $G$ is defined as a cumulative flux function.
Moreover, this definition allows for possibly discontinuous backward-delay functions $\eta$, as we will see that $G$ is uniquely determined as smooth function in the next section.

\subsection{Time-inhomogeneous renewal process}

The cumulative flux $G$ is instrumental in specifying the dynamics of the time-changed process $Y_\sigma$, whose density function solves the PDE problem \ref{def:qPDE}.
Given generic initial conditions, it is clear that $G$ shall only depend on the backward-delay functions $\eta$.
In the next section, we will make this $\eta$-dependence explicit by adapting results from elementary renewal analysis.
As a preliminary to this objective, we devote this section to exhibiting the renewal character of the time-changed dynamics.
In this perspective, let us introduce $\lbrace \tau_k \rbrace_{k \geq 1}$, the increasing sequence of reset times to be distinguished from the sequence of inactivation times $\{  \xi_k \}_{k \geq 0}$, with the convention that $\tau_0=0$.
These newly-introduced reset times can also be defined in terms of the backward-delay function $\eta$ as:

\begin{proposition}\label{prop:tauDef}
Given a backward-delay function $\eta$ in $\mathcal{W}$, the reset times $\lbrace \tau_k \rbrace_{k \geq 1}$ of the time-changed dynamics $Y_\sigma$ satisfy
\begin{eqnarray}\label{eq:tauDef}
\tau_k=\tau(\xi_k) \, , \quad \mathrm{with} \quad \tau(\xi) = \inf \{ \sigma > 0 \, \vert \, \xi(\sigma) = \sigma-\eta(\sigma) \geq \xi \} \, ,
\end{eqnarray}
where the forward function $\tau$ is the left-continuous generalized inverse of the nondecreasing function $\xi=\mathrm{id}-\eta$.
\end{proposition}

Again, just as for the function $\xi$, we will refer to the ``function $\tau$'' when  $\tau$ designates the forward function defined in \eqref{eq:tauDef} rather than a real variable.

\begin{proof}[Proof of Proposition \ref{prop:tauDef}]
This follows  from the fact that the delayed process $N_{\sigma-\eta(\sigma)}$ involved in \eqref{eq:Ysigma} results from the composition of the $\emph{c\`adl\`ag}$ counting process $N_t$ with the $\emph{c\`adl\`ag}$ nondecreasing function $\xi=\mathrm{id}-\eta$.
To see why, suppose that $Y_\sigma$ inactivates in $\xi_1$, i.e., that $N_\sigma$ has a jump discontinuity in $\xi_1$. 
Then, $Y_\sigma$ remains in zero for $\sigma>\xi_1$ until the first discontinuity time of the reset counting process $N_{\xi(\sigma)}$.
If $\xi^{-1}(\{\xi_1 \})=\{\tau_1 \}$ is a singleton, we set $\tau(\xi_1)=\tau_1$.
Otherwise,  $\xi^{-1}(\{\xi_1 \})$ is an interval including its left endpoint denoted by $\tau_1$.
By right-continuity of $N_\sigma$, the first discontinuity time of the composed process $N_{\xi(\sigma)}$  must be $\tau_1$, which is defined as: 
\begin{eqnarray*}
\tau_1 = \inf  \{ \sigma > 0 \, \vert \, \xi(\sigma) = \xi_1 \} = \inf  \{ \sigma > 0 \, \vert \, \xi(\sigma)  \geq \xi_1 \} \, .
\end{eqnarray*}
This justifies defining the function $\tau$  as the left-continuous generalized inverse of $\xi$.
\end{proof}

By  definition \eqref{eq:tauDef},  for all $k\geq 1$, $\tau_k$ satisfies $\tau_k-\xi_k \geq \eta(\tau_k)$ with equality if $\xi_k$ is a continuity point of $\tau$.
Otherwise, we can only say that $\tau_k-\xi_k \geq \eta(\tau_k)$. 
%
%
Thus the refractory period of $Y_\sigma$ does not necessarily coincide with the backward-delay function $\eta$ at the reset time.
However, if $\eta$ is uniformly bounded by $\Vert \eta \Vert_{0,\infty}$ on $\mathbbm{R}^+$, we have
\begin{eqnarray}\label{eq:tauIneq}
\tau(\xi) \leq  \inf \Big \{ \sigma > 0 \, \Big \vert \, \sigma- \sup_{  s \geq 0 }\eta(s) \geq \xi \Big \} =  \xi + \Vert \eta \Vert_{0,\infty} \, ,
\end{eqnarray}
Thus, in general, we have  $\eta(\tau_k) \leq \tau_k-\xi_k \leq  \Vert \eta \Vert_{0, \infty}$.

Clarifying the possible continuity issues of the time-changed dynamics motivates introducing one more delay function, the so-called forward-delay function defined by $\gamma=\tau-\mathrm{id}$.
By contrast with the backward-delay function, $\gamma$ allows us to consider  the refractory period as a function of the inactivation time: $\tau_k-\xi_k=\tau(\xi_k)-\xi_k=\gamma(\xi_k)$.
Backward and forward delay functions are naturally related via the following properties:

\begin{proposition}
$(i)$ For all backward delay functions $\eta$ in $\mathcal{W}$, the forward-delay function $\gamma$ is specified by:
\begin{eqnarray}\label{eq:charGamma}
\forall \; \xi \geq \xi_0 \, , \quad \gamma(\xi) = \inf \{ \sigma > 0 \, \vert \, \sigma \geq \eta(\sigma+\xi)\} \geq \nu \epsilon \, .
\end{eqnarray}

$(ii)$ Given two backward-delay function $\eta_a$ and $\eta_b$ in $\mathcal{W}$ with $\eta_a \geq \eta_b$, their corresponding forward-delay functions $\gamma_a$ and $\gamma_b$
satisfy $\gamma_a \geq \gamma_b$.
\end{proposition}

\begin{proof}
$(i)$ By definition of the forward function $\tau$ in \eqref{eq:tauDef}, for all $\xi \leq \xi_0$:
\begin{eqnarray*}
\gamma(\xi) 
=
 \tau(\xi) - \xi  
=
 \inf \{ \sigma > \xi \, \vert \, \sigma-\eta(\sigma) \geq \xi \} - \xi 
=
 \inf \{ \sigma > 0 \, \vert \, \sigma \geq \eta(\sigma+\xi)\}, .
\end{eqnarray*}
In particular, we have  $\gamma(\xi) \geq \inf_{\sigma \geq 0} \eta(\sigma)  \geq \nu \epsilon$.

$(ii)$ If $\eta_a \geq \eta_b$, we have $\{ \sigma > \, \vert \,\sigma \geq \eta_a(\sigma+\xi)  \} \subset \{ \sigma > \, \vert \,\sigma \geq \eta_b(\sigma+\xi)  \}$ so that by the characterization given in $(i)$, we have $\gamma_a \geq \gamma_b$.
\end{proof}

Notice that characterization \eqref{eq:charGamma} together with \eqref{eq:tauIneq} implies that $\Vert \eta \Vert_{0,\infty}=\Vert \gamma \Vert_{0,\infty}$.\\

It is now straightforward to exhibit the renewal character of the dynamics of $Y_\sigma$.
Unless stated otherwise, we assume the initial condition $p_{\sigma_0}=\delta_x$, $x>0$, for simplicity.
Given a backward-delay functions $\eta$ in $\mathcal{W}$, the sequences of times $\lbrace \xi_k \rbrace_{k \geq 0}$  and $\lbrace \tau_k \rbrace_{k \geq }$ are interwoven, i.e., $\xi_0<\tau_0=0<\xi_1 < \tau_1 <\xi_2 < \tau_2 \ldots$.
The refractory periods $\{ \tau_k-\xi_k \}_{k \geq 1}$ are determined by the forward-delay function $\tau_k-\xi_k=\gamma(\xi_k)$, which was precisely introduced to that end.
In between consecutive reset and inactivation times, the dynamics of the time-changed process $Y_\sigma$ is simply that of a Wiener process with unit negative drift.
Thus, the random variables $\{ \xi_{k+1}-\tau_{k} \}_{k \geq 0}$ are i.i.d. according to  $\Prob{\xi_{k+1}-\tau_{k} \leq \sigma}=H(\sigma,\Lambda)$ for all $k\geq1$ and  to $\Prob{ \tau_1 \leq \sigma \, \vert \, Y_0=x } =H(\sigma,x)$ otherwise, where $H$ denotes the first-passage cumulative distribution~\cite{Karatzas}
\begin{eqnarray*}
H(\sigma,x)  
= \frac{1}{2} 
\left( 
\mathrm{Erfc} \left( \frac{x-\sigma}{\sqrt{2 \sigma}}\right) + e^{2 x} \mathrm{Erfc} \left( \frac{x + \sigma}{\sqrt{2 \sigma}}\right) 
\right) \, .
\end{eqnarray*}
By convention, we set $H(\sigma,x)=0$ for $\sigma<0$, so that $H$ admits the density function
\begin{eqnarray*}
h(\sigma,x) = \partial_\sigma H(\sigma,x)  = \mathbbm{1}_{\{\sigma \geq 0 \}}\frac{xe^{-(x -\sigma)^2/2 \sigma}}{\sqrt{2 \pi \sigma^3}}  \, .
\end{eqnarray*}
In turn, the inter-inactivation epochs $\{ \xi_{k+1}-\xi_k \}_{k \geq 1}$ 
are independently distributed according to time-inhomogeneous distributions:
\begin{eqnarray*}
\Prob{\xi_{k+1} \leq \sigma \, \vert \, \xi_k}&=&\bar{H}_\Lambda(\sigma,\xi_k) = H \big(\sigma-\tau(\xi_k) , \Lambda \big)= H \big(\sigma-\xi_k-\gamma(\xi_k) , \Lambda \big) \, ,
\end{eqnarray*}
This shows that the sequence $\{  \xi_k \}_{k \geq 0}$ 
constitutes a time-inhomogenous renewal process.

Recognizing the renewal character of the time-changed dynamics $Y_\sigma$ suggests that its associated cumulative flux $G$ satisfies a renewal-type integral equation.
In order to establish this equation in the next section, we will need the following result, which shows that forward and backward functions $\xi$ and $\tau$ are well-behaved inverse functions of one another.

\begin{proposition}\label{prop:Peta}
Given a backward-delay function $\eta$ in $\mathcal{W}$,  for all $\sigma>0$, we have
\begin{eqnarray*}
\{ \xi > \xi_0 \, \vert \, \tau(\xi) = \xi+\gamma(\xi) \leq \sigma \} = \{ \xi > \xi_0 \, \vert \, \xi \leq \xi(\sigma) = \sigma-\eta(\sigma) \} \, .
\end{eqnarray*}
\end{proposition}

\begin{proof}
In order to prove the proposed set identity, we use the following characterization:
\begin{eqnarray}\label{eq:charSet}
\{ \xi' > \xi_0 \, \vert \, \tau(\xi')  \leq \sigma \}
&=&
\{ \xi' > \xi_0 \, \vert \, \inf \{ \tau'> \xi' \, \vert \, \xi(\tau') \geq \xi' \} \leq \sigma \} \, , \\
&=&
 {\displaystyle \cap}_{n \geq 1}\left\{ \xi' > \xi_0 \, \vert \, \exists \,  \tau'> \xi' \, , \; \xi(\tau') \geq \xi' , \tau' \leq \sigma+1/n \right\} \, , \\
 &=&
 {\displaystyle \cap}_{n \geq 1}\left\{ \xi' > \xi_0 \, \vert \, \exists \,  \tau'> \xi' \, , \; \xi(\tau') \geq \xi' , \sigma \leq \tau' \leq \sigma+1/n \right\} \, , 
\end{eqnarray}
where the last equality follows from the fact that $\xi = \mathrm{id}-\eta$ is nondecreasing.

Consider $\xi' > \xi_0$ such that $\xi' \leq \xi(\sigma)$, then for all $n \geq 1$, take any $\tau'$ such that $\sigma \leq \tau' \leq \sigma+1/n$, we have $\xi' \leq \xi(\sigma) = \sigma-\eta(\sigma) < \sigma \leq \tau'$ and $\xi(\tau') \geq \xi(\sigma) \geq \xi'$.
Thus $ \{ \xi' > \xi_0 \, \vert \, \xi' \leq \xi(\sigma) \} \subset \{ \xi' > \xi_0 \, \vert \, \tau(\xi')  \leq \sigma \}$.

Reciprocally, consider $\xi' > \xi_0$ such that $\tau(\xi')  \leq \sigma$. Then by characterization \eqref{eq:charSet}, there is a sequence $\{ \tau_n \}_{n\geq 1}$ such that $\xi(\tau_n) \geq \xi'$ and $\sigma \leq \tau_n  \leq \sigma+1/n$.
Suppose that $\xi'> \xi(\sigma)$, we then have
\begin{eqnarray*}
\xi(\sigma)< \xi' \leq  \liminf_{n \to \infty} \xi(\tau_n) = \lim_{\tau \to \sigma^+} \xi(\tau) \, ,
\end{eqnarray*}
which contradicts the right continuity of $\xi$. 
Thus we must have $\xi' \leq \xi(\sigma)$, which means that $\{ \xi' > \xi_0 \, \vert \, \tau(\xi')  \leq \sigma \} \subset \{ \xi' > \xi_0 \, \vert \, \xi' \leq \xi(\sigma) \}$.
\end{proof}

A direct consequence of the above proposition is that $\Prob{\tau(\xi)  \leq \sigma } = \Prob{\xi \leq \sigma-\eta(\sigma)}$. This result will feature prominently in establishing a renewal-type equation for the cumulative flux function $G$.

\subsection{Quasi-renewal equation}

Here, our goal is to adapt elementary results from renewal analysis to characterize the cumulative flux $G$ as the unique (smooth) solution of a renewal-type equation.
For the elementary initial condition $q_0=\delta_x$,
this renewal-type equation can be deduced from the representation of $G$ as a probabilistic series
\begin{eqnarray}\label{eq:Ftx}
G( \sigma)  = \Exp{\sum_{k=1}^{\infty}  \mathbbm{1}_{ \left\{ \xi_k <  \sigma \right\} } \, \bigg \vert \, Y_0=x } = \sum_{k=1}^{\infty} \Prob{ \xi_k \leq  \sigma \, \vert \, Y_0=x}  \, .
\end{eqnarray}
Observe that by nonnegativity of forward-delay functions $\gamma$, each of the probabilities involved in the series is upper bounded by its counterpart in the compactly converging series representation without delay: $\Prob{ \xi_k \leq  \sigma \, \vert \, Y_0=x}\leq H(\sigma, x+(k-1)\Lambda)$.
This justifies the validity of the series representation \eqref{eq:Ftx}.
Moreover, by divisibility of the first-passage distribution, one can check that 
\begin{eqnarray*}
 \sum_{k=1}^\infty H(\sigma, x+(k-1)\Lambda) =  H(\sigma, x)  + \int H(\sigma-\tau, \Lambda) \sum_{k=1}^\infty H(\tau, x+(k-1)\Lambda) \, \dd \tau \, ,
\end{eqnarray*}
yielding the classical renewal integral equation satisfied by $G$ in the absence of delays. We extend  this result for nonzero delays in the following proposition:

\begin{proposition}\label{prop:Renewal}
Given a backward-delay function $\eta$ in $\mathcal{W}$, the cumulative flux function $G$ associated to the PDE problem \ref{def:qPDE} is the unique solution of the renewal-type equation:
\begin{eqnarray}\label{eq:DuHamel}
G(\sigma) = \int_0^\infty H(\sigma,x) q_0(x) \, \dd x  + \int_0^\sigma H(\sigma-\tau,\Lambda) \, \dd G(\tau-\eta(\tau))   \, .
\end{eqnarray}
\end{proposition}

\begin{proof}
It is enough to show the result for the elementary initial condition $q_0=\delta_x$.
Let us consider $\xi_{k+1}$, the $(k+1)$-th inactivation time of $Y_\sigma$, which is necessarily preceded by the $k$-th reset time: $\xi_{k+1} > \tau_k > 0$.
Conditioning on $\tau_k$ for $k\leq 1$ yields
\begin{eqnarray*}
 \Prob{ \xi_{k+1} \leq  \sigma \, \vert \, Y_0=x} 
 = 
 \Exp{\Prob{\xi_{k+1} \leq \sigma \, \vert \, \tau_k } \, \vert \, Y_0=x}
  = 
 \Exp{H(\sigma - \tau_k, \Lambda)\, \vert \, Y_0=x} \, ,
\end{eqnarray*}
where the last equality uses the fact that $\{ \xi_{k+1}-\tau_k \}_{k \geq 1}$ are i.i.d. according to the distribution $H(\cdot, \Lambda)$.
Thus, singling out the first term, the series representation of $G$ given in \eqref{eq:Ftx} reads 
\begin{eqnarray}
G( \sigma )  
&=&
\Prob{ \xi_1 \leq  \sigma \, \vert \, Y_0=x}
+
\sum_{k=2}^{\infty}   \Exp{H(\sigma - \tau_{k-1}, \Lambda)\, \vert \, Y_0=x} \, , \nonumber\\
&=&
H(\sigma,x)\label{eq:Gsx}
+  \int_0^\sigma H(\sigma - \tau, \Lambda) \,\sum_{k=1}^{\infty}  \dd \Prob{\tau_k \leq \tau \, \vert \, Y_0=x}  \, .  
\end{eqnarray}
We conclude by  expressing the series appearing as an integrator function above in terms of $G$.
Invoking Proposition \ref{prop:Peta}, we have the compact convergence
\begin{eqnarray*}
\sum_{k=1}^{\infty} \Prob{ \tau_k \leq  \sigma \, \vert \, Y_0=x}
=
 \sum_{k=1}^{\infty} \Prob{ \xi_k \leq  \sigma -\eta(\sigma) \, \vert \, Y_0=x} 
 =
 G(\sigma-\eta(\sigma)) \, ,
\end{eqnarray*}
which leads to \eqref{eq:DuHamel} upon substitution in \eqref{eq:Gsx}.

To show uniqueness, suppose $G_1$ and $G_2$ both solves \eqref{eq:DuHamel}. 
Then, $G_1$ and $G_2$ are necessarily smooth functions with derivatives $g_1$ and $g_2$ satisfying 
\begin{eqnarray*}
g_1(\sigma)-g_2(\sigma)
= \int_0^\sigma  h(\sigma-\tau,\Lambda) \dd \big[ G_1(\xi(\tau)) -G_2(\xi(\tau))) \, .
\end{eqnarray*}
Introducing the function $\tau(\xi)$ allows on to perform the generalized change of variable $\xi=\xi(\tau)$.
For small enough $\sigma >0$, such a change of variable yields:
 \begin{eqnarray*}
\vert g_1(\sigma)-g_2(\sigma) \vert
&=& \bigg \vert \int_0^\sigma  h(\sigma-\tau(\xi),\Lambda) \dd \big[ G_1(\xi) -G_2(\xi)) \bigg \vert \, , \\
&\leq& \int_0^\sigma  h(\sigma-\tau(\xi),\Lambda) \big \vert g_1(\xi) -g_2(\xi)) \big \vert \, \dd \xi \, , \\
&\leq&  \left( \int_0^\sigma  h(\sigma-\tau(\xi),\Lambda)  \dd \xi \right)  \sup_{0 \leq \xi \leq \sigma} \vert g_1(\xi)-g_2(\xi) \vert  \, , \\
&\leq& \sigma \Vert h(\cdot,\Lambda) \Vert_\infty \vert g_1(\xi)-g_2(\xi) \vert  \, , 
\end{eqnarray*}
where $\Vert h(\cdot,\Lambda) \Vert_\infty$ denotes the finite infinity norm of $\sigma \mapsto h(\sigma,\Lambda)$.
This inequality establishes that $g_1=g_2$ on any interval $[0,\sigma)$ with $\sigma< 1/\Vert h(\cdot,\Lambda) \Vert_\infty$.
This local uniqueness result transfers to $G_1$ and $G_2$ by virtue of $G_1(0)=G_2(0)=0$.
Finally, global uniqueness can be recovered by standard methods of continuation.
\end{proof}

The above result makes it clear that as a solution to \eqref{eq:DuHamel}, the cumulative flux function $G$ inherits all the regularity properties of $\sigma \mapsto H( \sigma, \Lambda)$, i.e., $G$ is a smooth function for $\sigma>0$.
Moreover, with $G$ specified as a solution to \eqref{eq:DuHamel}, the full solution of the inhomogeneous PDE \eqref{eq:qPDE} can be expressed in terms of the corresponding homogeneous solutions by Duhamel's principle.
These homogeneous solutions are known in closed form~\cite{Karatzas}:
\begin{eqnarray*}
\kappa(\sigma,y,x) = \frac{e^{-\frac{(y - x +  \sigma)^2}{2\sigma}}}{\sqrt{2 \pi \sigma}}  \left(1-e^{ -\frac{2 x y}{\sigma}} \right) \, .
\end{eqnarray*}
Thus, Proposition \ref{prop:Renewal} admits the following corollary:

\begin{corollary}
Given a backward-delay function $\eta$ in $\mathcal{W}$ and normalized initial conditions $(q_0, g_0)$  in $\mathcal{M}(\mathbbm{R}^+) \times \mathcal{M}([\xi_0,0))$, there is a unique solution density function $(\sigma,x) \mapsto q(\sigma,x)=\dd \Prob{0<Y_\sigma \leq x} / \dd x$  to the PDE problem \ref{def:qPDE} which admits the integral representation
\begin{eqnarray*}
q(\sigma,y) = \int_0^\infty \kappa(\sigma,y,x) q_0(x) \, \dd x  + \int_0^\sigma \kappa(\sigma-\tau,y,\Lambda) \, \dd G(\tau-\eta(\tau))   \, ,
\end{eqnarray*}
where $G$ is the solution to the renewal-type equation \eqref{eq:DuHamel}. 
\end{corollary}

The above renewal analysis has allowed us to justify the existence, uniqueness, and regularity of time-changed dynamics assuming the backward-delay function $\eta$ known.
However, $\eta$ is actually an unknown of the problem for being ultimately defined in term of the time change $\Phi$ via Definition \ref{def:Eta}.
We devote the next section to exhibiting under which conditions $\eta[\Phi]$ parametrizes an admissible dPMF dynamics.

\section{Fixed-point problem}\label{sec:fixedPoint}

In this section, we establish our main result, Theorem \ref{th:fixedpoint_int}, by establishing that dPMF dynamics are equivalent to certain constitutively well-posed time-changed dynamics.
To do so, for any candidate time change $\Phi$ in $\mathcal{T}$, we consider the time-changed dynamics $Y_\sigma$ that is uniquely defined by the backward delay function $\eta[\Phi]$ in $\mathcal{W}$.
Then, we look for possibly explosive, weak solutions to our original PDE problem \ref{def:delayedPDE} under the form $(t,x) \mapsto q[\Phi](\Phi(t),\sigma)$, where $q[\Phi]$ is the density function of the process $Y_\sigma$.
This leads to exhibiting a natural condition on $\Phi$ to parametrize such a weak solution, as stated in the following propositon:

\begin{proposition}\label{prop:qPDE}
Given a time change $\Phi$ in $\mathcal{T}$, let $q$ be the solution of the time-changed PDE problem \ref{def:qPDE} associated with $\eta[\Phi]$ in $\mathcal{W}$ on $\mathbbm{R}^+ \times \mathbbm{R}^+$.
Then, $p(t,x)=q(\Phi(t),x)$ with $\Phi$ a valid time change in $\mathcal{T}$, weakly solves the PDE problem \ref{def:delayedPDE} on $\mathbbm{R}^+ \times \mathbbm{R}^+$ if and only if 
\begin{eqnarray}\label{eq:condEq}
\forall \; t \geq 0 \, , \quad \Phi(t)=\nu t+\lambda G(\Phi(t)) \, .
\end{eqnarray}
In particular, the cumulative flux function of the dPMF dynamics is given by $F=G \circ \Phi$.
\end{proposition}

\begin{proof}

We proceed in two steps: $(i)$ we characterize the time-changed dynamics associated with $\eta[\Phi]$ as a solution to a weak PDE problem, $(ii)$ we show that this weak formulation is equivalent to that of Definition \ref{def:weakPDE} if and only if the proposed criterion holds.

$(i)$ As a solution to the PDE problem \ref{def:qPDE}, $q$ is also a weak solution in the usual sense:
for all $S>0$ and for all test functions $v$ in $C_{0,\infty}([-\xi_0,S)\times \mathbbm{R})$, we have
\begin{eqnarray*}
\int_0^S \int_0^\infty \left[  \mathcal{L}[q](\sigma,x) + \frac{\dd}{\dd \sigma} [G(\sigma-\eta(\sigma))] \delta_{\Lambda} - \partial_\sigma q(\sigma,x) \right] v(\sigma,x) \, \dd \sigma \dd x \, .
\end{eqnarray*}
In turn, performing integration by parts in the distributional sense yields
\begin{eqnarray*}
\int_{0}^\infty \big[ q(S,x)v(S,x)-q(0,x)v(0,x) \big] \, \dd x 
&=& \\
&&\hspace{-140pt} \int_0^S\int_{0}^\infty  q(\sigma,x) \big[\mathcal{L}^\dagger[v](\sigma,x)+ \partial_\sigma v(\sigma,x) \big] \, \dd x \,  \dd \sigma \nonumber\\
&& \hspace{-140pt}+ \int_0^S v(\sigma,\Lambda)  \, \dd G(\sigma-\eta(\sigma))  - \int_0^S   v(\sigma,0) \, \dd G(\sigma) \, .
\end{eqnarray*}
Our goal is to perform the change of variable $\sigma=\Phi(t)$ to recover the weak form of solutions from Definition \ref{def:weakPDE}.
Given any nondecreasing function $M$ and $N$ with $N$ right-continuous, for any bounded Borel measurable function $f$, the substitution formula reads 
 \cite{Falkner:2012}
\begin{eqnarray*}
\int_a^b f(x)\, \dd (N\circ M)(x) = \int_{M(a)}^{M(b)} f(X(y))\, \dd N(y) \, ,
\end{eqnarray*}
where $X$ is the left-continuous generalized inverse of $M$: $X(y)=\inf \{ x \in [a,b] \vert y \leq M(x) \}$.
Moreover, if $N$ is continuous, $X$ can be any generalized inverse of $M$.
Thus, we have 
\begin{eqnarray*}
\int_0^S v(\sigma,\Lambda) \, \dd G(\sigma-\eta(\sigma)) 
&=& \int_0^S v(\sigma,\Lambda) \, \dd (G \circ \Phi)(\Psi(\sigma)-\epsilon) \, , \\
&=& \int_0^S v(\Phi(t^-),\Lambda) \, \dd (G \circ \Phi)(t-\epsilon) \, .
\end{eqnarray*}
Therefore, defining $u(t,x)=v(\Phi(t),x)$ and $T=\Psi(S)$, we have
\begin{eqnarray*}
\int_{0}^\infty \big[ q(\Phi(T),x)u(T,x)-q(\Phi(0),x)u(0,x) \big] \, \dd x 
&=& \\
&& \hspace{-240pt} \int_0^T \int_{0}^\infty  q(\Phi(t^-),x) \big[\mathcal{L}^\dagger[u](t^-,x)+ \partial_\sigma v(\Phi(t^-),x) \big] \, \dd x \,   \dd \Phi(t) \nonumber\\
&& \hspace{-240pt}+ \int_0^T u(t^-,\Lambda)  \, \dd (G \circ \Phi)(t-\epsilon)  - \int_0^T    u(t^-,0) \, \dd (G \circ \Phi)(t) \, .
\end{eqnarray*}
Given a smooth test function $v$, for all $x\geq 0$, the function $t \mapsto u(t,x)=v(\Phi(t),x)$ is a function with bounded variation.
Accordingly, we can apply the generalized chain rule involving Vol'pert superposition principle \cite{Volpert:1967,DalMaso:1991} to obtain
\begin{eqnarray*}
\frac{\dd u}{\dd t}(\cdot ,x)=  \frac{\dd}{\dd t}[v(\Phi(\cdot),x)] = \partial_\sigma\hat{v}(\Phi,x)(\cdot) \frac{\dd \Phi}{\dd t} \, ,
\end{eqnarray*}
where  $\partial_\sigma\hat{v}$ denotes the average superposition of $\partial_\sigma v$:
\begin{eqnarray*}
\partial_\sigma \hat{v}(\Phi,x)(t)=\int_0^1 \partial_\sigma v\big(\Phi(t^-)+z(\Phi(t)-\Phi(t^-))\big) \, \dd z \, .
\end{eqnarray*}
Now, we can always restrict our choice of test functions $v$ to these functions in $C_{0,\infty}([-\epsilon,\infty)\times \mathbbm{R})$ such that for all $t$ in the discontinuity set $\mathcal{D}_{\Phi}$, we have 
\begin{eqnarray*}
v(\Phi(t^-))=v(\Phi(t)) \quad \mathrm{and} \quad \partial_\sigma v(\Phi(t^-))=\partial_\sigma v(\Phi(t)) = \partial_\sigma \hat{v}(\Phi,x)(t) \, ,
\end{eqnarray*}
so that $t \mapsto u(x,t)$ is continuously differentiable. This means in particular that for all test functions $u$ in $C_{0,\infty}([-\epsilon,S)\times \mathbbm{R})$, we have
\begin{eqnarray}\label{eq:uWeak}
\int_{0}^\infty \big[ q(\Phi(T),x)u(T,x)-q(\Phi(0),x)u(0,x) \big] \, \dd x 
&=& \\
&& \hspace{-240pt} \int_0^T \int_{0}^\infty  q(\Phi(t^-),x) \mathcal{L}^\dagger[u](t,x) \, \dd \Phi(t)  + \int_0^T \int_{0}^\infty q(\Phi(t^-),x) \partial_t u(t,x) \, \dd x \,  \dd t \nonumber\\
&& \hspace{-240pt}+ \int_0^T u(t,\Lambda)  \, \dd (G \circ \Phi)(t-\epsilon)  - \int_0^T    u(t,0)\, \dd (G \circ \Phi)(t)  \nonumber\, .
\end{eqnarray}

$(ii)$ Let us now prove the proposition.
If $\Phi(t)=\nu t+\lambda G(\Phi(t))$, we directly see that substituting $ \dd \Phi(t)=\nu \, \dd t + \lambda \, \dd (G\circ \Phi)(t)$ in the above equation shows that $(t,x) \mapsto p(t,x)=q(\Phi(t),x)$ is a weak solution of PDE problem \ref{def:delayedPDE} for $F=G \circ \Phi$.

Reciprocally, let  $(t,x) \mapsto  p(t,x)=q(\Phi(t),x) $ be a weak solution of the PDE problem \ref{def:delayedPDE}.
Specifying  the defining property of $p$ as a weak solution for $F$ with $u(x,t)=1$ yields 
\begin{eqnarray*}
\int_0^T \int_{0}^\infty \partial_t p(t,x) u(t)\, \dd x \, \dd t = \int_0^T u(t)  \, \dd F(t-\epsilon)  - \int_0^T   u(t) dF(t) \, ,
\end{eqnarray*}
whereas specifying \eqref{eq:uWeak} for $u(x,t)=1$ yields an alternative expression for the same quantity
\begin{eqnarray*}
\int_0^T \int_{0}^\infty \partial_t q(\Phi(t),x) u(t)\, \dd x \, \dd t = \int_0^T u(t)  \, \dd (G \circ \Phi)(t-\epsilon)  - \int_0^T   u(t) \, \dd (G \circ \Phi)(t) \, .
\end{eqnarray*}
Moreover, for $\Psi$ being the right-continuous generalized inverse of the strictly increasing function $\Phi,$ we have $\Psi \circ \Phi=\mathrm{id}$  on $\mathbbm{R}^+$  and the initial condition $G=F\circ \Psi$ on $(-\eta(\epsilon),0]$ implies that $G \circ \Phi=F$ on $[-\epsilon,0)$.
Then, by the same reasoning as in Proposition \ref{prop:uniqueF}, $F(t)=(G \circ \Phi)(t)$ for all $t \geq 0$.
In turn, subtracting \eqref{eq:uWeak} with $q(\Phi(t),x)$ identified to $p(t,x)$ from equation \eqref{eq:weakPDE}  with $F$ identified to $G \circ \Phi$ yields
\begin{eqnarray*}
\int_0^T \int_{0}^\infty  p(t^-,x) \mathcal{L}^\dagger[u](t,x)\, \dd x \, \dd \Phi(t) 
=
\int_0^T \int_{0}^\infty  p(t^-,x) \mathcal{L}^\dagger[u](t,x)\, \dd x \,   \dd \big[\nu t +\lambda F(t)\big] \, .
\end{eqnarray*}
It is clear that the above equality holds for all smooth  functions such that $\mathcal{L}^\dagger[u]$ remains bounded on $[0,T)\times (0,\infty)$.
This observation allows us to  specify the above equation for test functions of the form $u(x,t)=x w(t)$ with $w$ in $C_{0,\infty}([0,T))$, so that we obtain
\begin{eqnarray*}
\int_0^T(1-r(t)) w(t) \, \dd \big[\nu t +\lambda F(t)-\Phi(t) \big] = 0 \, ,
\end{eqnarray*}
where $r(t)$ is the probability of being in the refractory state at $t$.
As for all locally bounded backward-delay functions $\eta$, $r(t)>0$ remains bounded away from one, and both $\Phi$ and $F$ are \emph{c\`adl\`ag} functions, this implies that 
\begin{eqnarray*}
\forall t \geq 0 \, , \quad \Phi(t) = \nu t +\lambda F(t) = \nu t+\lambda G(\Phi(t)) \, ,
\end{eqnarray*}
which concludes the proof.
\end{proof}

Proposition \ref{prop:qPDE} shows that the existence and uniqueness of dPMF dynamics amounts to the existence and uniqueness of a time change $\Phi$ in $\mathcal{T}$ solving \eqref{eq:condEq}.
In this light, \eqref{eq:condEq} appears as a self-consistent relation rather than a mere definition as in \eqref{eq:Phi}.
In fact, \eqref{eq:condEq} defines a fixed-point problem satisfied by these time changes $\Phi$ that parametrized dPMF dynamics.
The fixed-point nature of \eqref{eq:condEq} follows from the fact that the cumulative flux function $G$ involved in \eqref{eq:condEq} is functionally dependent on $\Phi$ via $\eta$.

To account for blowups, solutions $\Phi$ to  \eqref{eq:condEq} are allowed to be discontinuous, which leads to possible degeneracy issues.
Indeed, as the cumulative flux $G$ must be smooth, $\Phi$ can only become discontinuous if  \eqref{eq:condEq} is degenerate in the sense that $\sigma = \nu t + G(\sigma)$ admits multiple solutions $\sigma$ for some times $t$. 
These times at which \eqref{eq:condEq} becomes degenerate will actually mark the occurrence of blowups in the original dPMF dynamics.
To avoid dealing with such degeneracies, it is actually desirable to reformulate the fixed-point problem in terms of the inverse time change.
This is because by contrast with the possibly discontinuous increasing function $\Phi$, the inverse time change $\Psi=\Phi^{-1}$ is defined as a continuous nondecreasing function.
Such a formulation yields our main result stated in Theorem \ref{th:fixedpoint_int}, which directly follows from the following proposition:

\begin{proposition}\label{def:fixedPoint}
The inverse time-change $\Psi$ parametrizes a dPMF dynamics if and only if it solves the fixed-point problem
\begin{eqnarray}\label{eq:fixedPoint}
\forall \; \sigma \geq 0 \, , \quad \Psi(\sigma) = \sup_{0 \leq \xi \leq \sigma} \big( \xi - \lambda G[\Psi](\xi)\big) / \nu \, .
\end{eqnarray}
\end{proposition}

\begin{proof}
Consider $\Phi$ in $\mathcal{T}$ solving \eqref{eq:condEq}. Then, for all $\sigma \geq 0$, writing $\sigma=\Phi(t)$ and $t=\Psi(\sigma)$ in \eqref{eq:condEq}, we have 
\begin{eqnarray*}
\Psi(\sigma) = \big( \sigma - \lambda G[\Psi](\sigma)\big) / \nu = \sup_{0 \leq \xi \leq \sigma} \big( \xi - \lambda G[\Psi](\xi)\big) / \nu \, ,
\end{eqnarray*}
where the last equality follows from $\Psi=\Phi^{-1}$ being \emph{c\`adl\`ag} increasing by definition of $\mathcal{T}$.

Reciprocally, consider $\Phi$ solving the fixed-point problem \eqref{eq:fixedPoint}.
Then, $\Phi$ is necessarily a nonnegative, \emph{c\`adl\`ag}, nondecreasing function for being defined as a running maximum function with $\Psi(0)=0$.
Moreover, the function $\sigma \mapsto \eta(\sigma)= \sigma - \Psi^{-1}(\Psi(\sigma)-\epsilon)$ is a nonnegative, \emph{c\`adl\`ag} function so that $G$ is a well-defined cumulative function satisfying the renewal-type equation \eqref{eq:DuHamel}. 
In particular $G$ is a nondecreasing smooth function so that $w_\Psi$ satisfies $w_\Psi \leq 1/\nu$. This shows that $\Phi=\Psi^{-1}$ belongs to $\mathcal{T}$.
Finally, we conclude by observing that
\begin{eqnarray*}
\Phi(t) 
&=& \inf  \Big \{ \sigma \geq 0 \, \Big \vert \, \Psi(\sigma) > t  \Big  \} \, , \\
&=& \inf \Big \{ \sigma \geq 0 \, \Big \vert \,\sup_{0 \leq \xi \leq \sigma} \big( \xi - \lambda G[\Psi](\xi)\big) / \nu > t \Big \} \, , \\
&=& \inf \Big \{ \sigma \geq 0 \, \Big \vert \, \big( \sigma - \lambda G[\Psi](\sigma)\big) / \nu > t \Big \} \, .
\end{eqnarray*}
Thus, by continuity of $G$, we have
\begin{eqnarray*}
\Phi(t) = \inf \Big \{ \sigma \geq 0 \, \Big \vert \, \big( \sigma - \lambda G[\Psi](\sigma)\big) / \nu = t \Big \} \, , 
\end{eqnarray*}
so that when $\Phi(t)$ exists, it necessarily satisfies \eqref{eq:condEq}.
\end{proof}

The above proposition proves our main result stated in the introduction as Theorem \ref{th:fixedpoint_int}. 
The next section establishes its practical usefulness by showing that the corresponding fixed-point problem admits solutions parametrizing explosive dPMF dynamics.

\section{Local blowup solutions}\label{sec:local}

In this section, we establish that for large enough interaction parameters, the fixed-point problem \ref{def:fixedPoint} locally admits solutions with analytically well-defined blowups.
To show this, we first define a contracting, regularized fixed-point map $\mathcal{F}_\delta$ over an appropriately chosen Banach space of candidate functions.
We then show that $\Psi = \lim_{\delta \to 0^+} \Psi_\delta$, where $\Psi_\delta$ uniquely solves the fixed-point equation $\Psi=\mathcal{F}_\delta[\Psi]$, defines locally the unique maximum smooth inverse time change up to the first putative blowup.
Finally, we analytically resolve the ensuing blowup episode by interpreting blowup  in the time-changed picture as linear dynamics with absorption but without reset.

\subsection{Regularized fixed-point problem}

The direct resolution of the dPMF fixed-point problem \ref{def:fixedPoint}  is not possible by standard analysis when allowing for discontinuous time changes.
To remedy this point, we consider a set of regularized fixed-point problems which approximates the original one, but for which delay functions will be continuous.
These approximate problems are defined on the following restricted space of candidate solutions:

\begin{definition}
Given a real $\delta$ such that $0<\delta<1/\nu$ and for all $\xi>0$, we restrict the candidate space for inverse time changes to 
\begin{eqnarray*}
\mathcal{C}_\delta([\xi_0,\xi])) = \left\{ \Psi \in C_0([\xi_0, \xi]) \, \Bigg \vert \, 
\begin{array}{ccc}
 \Psi(\xi)=\Psi_0(\xi)  &  ,  &   \xi_0 \leq \xi \leq 0  \vspace{5pt} \\
 \delta \leq w_\Psi(y,x)  \leq 1/\nu  & ,  &  0 \leq x , y \leq \xi 
\end{array}
\right\} \, ,
\end{eqnarray*}
where $w_\Psi$ is the difference quotient $w_\Psi(y,x) = (\Psi(y) - \Psi(x))/(y-x)$.
\end{definition}

For all $\xi> 0$, the candidate space $\mathcal{C}_\delta([\xi_0,\xi]))$ is a Banach space with respect to the uniform norm, denoted by $\Vert \cdot \Vert_{\xi_0,\xi}$.
Every candidate functions $\Psi$ in $\mathcal{C}_\delta([\xi_0,\infty))$ can naturally serve as an inverse time change, i.e.,  $\Psi^{-1}$ belongs to $\mathcal{T}$ .
Moreover, choosing $\delta>0$ enforces that every $\Psi$ in $\mathcal{C}_\delta([\xi_0,\infty))$ is a strictly increasing, continuous function on $\mathbbm{R}^+$, so that $\Psi^{-1}$ is also  a strictly increasing, continuous function on $\mathbbm{R}^+$ with difference quotient satisfying $\nu \leq w_{\Psi^{-1}}(y,x) \leq 1/\delta$.
As a result, the forward function $\tau[\Psi]$  is continuous in $C([\xi_0,\infty))$ and so is the backward function $\xi[\Psi]$ in $C([0,\infty))$, being defined as:
\begin{eqnarray*}
\tau[\Psi](\xi) = \Psi^{-1}\big(\Psi(\xi)+\epsilon \big)  \quad \mathrm{and} \quad
\xi[\Psi](\sigma) = \Psi^{-1}\big(\Psi(\sigma)-\epsilon \big) \, . 
\end{eqnarray*}
In the absence of discontinuities, we have the equivalence $\tau=\tau[\Psi](\xi) \Leftrightarrow \xi[\Psi](\tau)=\xi$, which implies that $\gamma[\Psi](\xi)=\eta[\Psi](\tau)$.
The  key reason to introduce the space $\mathcal{C}_\delta([\xi_0,\xi]))$  is the following uniform Lipschitz property:

\begin{proposition}\label{prop:Lipsc}
For all $\Psi$ in $\mathcal{C}_\delta([\xi_0,\infty))$, the map $\Psi \mapsto \xi[\Psi]$ are $1/\delta$-Lipschitz with respect to the uniform norm on $\mathcal{C}_\delta([\xi_0,\infty))$. \end{proposition}

\begin{proof}
Every function $\Psi$ in $\mathcal{C}_\delta([\xi_0,\xi]))$ admits a continuous inverse function $\Psi^{-1}$ on $\mathbbm{R}^+$ with bounded difference quotient such that $\nu \leq w_{\Psi^{-1}}\leq 1/\delta$.
Therefore, for all $\sigma>0$:
\begin{eqnarray*}
\vert \xi[\Psi_a](\sigma)-\xi[\Psi_b](\sigma) \vert &=& \vert \Psi^{-1}_a(\Psi_a(\sigma) - \epsilon) - \Psi^{-1}_b(\Psi_b(\sigma) - \epsilon) \vert \, \\
& \leq & \vert \Psi_a(\sigma)-\Psi_b(\sigma) \vert /\delta \, .
\end{eqnarray*}
\end{proof}

Our goal is to formulate the dPMF fixed-point problem \ref{def:fixedPoint} on the candidate Banach space $\mathcal{C}_\delta([\xi_0,\infty))$.
However, it turns out that the natural fixed-point map
\begin{eqnarray*}
\Psi \mapsto \Big\{ \sigma \mapsto \sup_{\xi_0 \leq \xi \leq \sigma}\big(\xi-\lambda G[\Psi](\xi)\big) / \nu \Big\} \, .
\end{eqnarray*}
does not stabilize  $\mathcal{C}_\delta([\xi_0,\infty))$ in the sense that it can produce functions with difference quotient below $\delta$.
To define a fixed-point map that stabilizes $\mathcal{C}_\delta([\xi_0,\infty))$, we need to introduce the function $\mathcal{S}_\delta: \mathcal{C}_{-\infty}([\xi_0,\infty)) \rightarrow \mathcal{C}_\delta([\xi_0,\infty))$ defined by
\begin{eqnarray*}
\mathcal{S}_\delta[\psi](\sigma) = \sup_{0 \leq s \leq \sigma} \left\{ \psi(s)-\delta s \right\} + \delta \sigma \, .
\end{eqnarray*}
The stabilizing role of $\mathcal{S}_\delta$ follows from noticing that for all $0 \leq x \leq y$, we have
\begin{eqnarray*}
\mathcal{S}_\delta[\psi](y)-\mathcal{S}_\delta[\psi](x) = \! \sup_{0 \leq s \leq y} \left\{ \psi(s)-\delta s \right\} - \! \sup_{0 \leq s \leq x} \left\{ \psi(s)-\delta s \right\} + \delta (y-x) \geq \delta (y-x) \, .
\end{eqnarray*}
This shows that $w_{\mathcal{S}_\delta[\psi]} \geq \delta$, so that $\mathcal{S}_\delta$ stabilizes $\mathcal{C}_\delta([\xi_0,\infty))$ from below.
In turn, we can show that composing the natural fixed-point map of  Definition \ref{def:fixedPoint} with $\mathcal{S}_\delta$ defines a well-posed map on $\mathcal{C}_\delta([\xi_0,\infty))$.

With these conventions, the regularized fixed-point map is specified as follows.


\begin{proposition}\label{def:fixedPoint2}
Given $\Psi$ in $\mathcal{C}_\delta([\xi_0,\infty))$ and initial conditions $(q_0,g_0)$, setting   
\begin{eqnarray*}
\mathcal{F}_\delta [\Psi] &=&  \mathcal{S}_\delta [\psi] \quad \mathrm{with} \quad 
\psi (\sigma)=\left\{
\begin{array}{ccc}
\displaystyle  \frac{1}{\nu}\big( \sigma - \lambda G[\Psi](\sigma)  \big) &\: \mathrm{if} \: &  \sigma \geq 0 \, , \\
\Psi_0(\sigma)  & \: \mathrm{if} \: &   \xi_0 \leq \sigma < 0 \, .
\end{array}
\right. 
\end{eqnarray*}
defines a map from $\mathcal{C}_\delta([\xi_0,\infty))$ to $\mathcal{C}_\delta([\xi_0,\infty))$ such that $\psi$ is a smooth function on $(0,\infty)$.
\end{proposition}

\begin{proof}
Let us check that for all $\Psi$ in $\mathcal{C}_\delta([\xi_0,\infty))$, $\mathcal{F}_\delta [\Psi]$ is also in $\mathcal{C}_\delta([\xi_0,\infty))$.
Given $\Psi$ in $\mathcal{C}_\delta([\xi_0,\infty))$, $\gamma[\Psi]$ is a positive, continuous, bounded function, which represents a forward-delay function compatible with the initial condition $ g_0$ on $[\xi_0,0)$.
In turn, we can interpret  $G[\Psi]$ as the associated cumulative flux, which is a smooth function on $(0,\infty)$ for satisfying the renewal-type equation \eqref{eq:DuHamel}.
In particular, $G[\Psi]$ is continuously differentiable on $C((0,\infty))$, with positive derivative denoted $g[\Psi]$.
Then, the definition of $\psi$ in terms of $G[\Psi]$ implies that $\psi$ belongs to  $\mathcal{C}_{-\infty}([\xi_0,\infty))$, so that $\Psi=\mathcal{S}_\delta[\Psi]$  belongs to  $\mathcal{C}_{\delta}([\xi_0,\infty))$.
\end{proof}

%
%

\subsection{Contraction argument}


We establish in the following proposition that the mapping $\mathcal{F}_\delta$ from Proposition \ref{def:fixedPoint2} is a contraction on the Banach spaces $\mathcal{C}_{\delta}([\xi_0,\sigma))$ for small enough $\sigma>0$. 
The proof will rely on the fact that for $\sigma>0$ smaller than the time-wrapped refractory periods, the mapping $\mathcal{F}_\delta$ loses its renewal character. 
Then, the Lipschitz continuity of the  mapping $\Psi \mapsto \xi[\Psi]$ on $\mathcal{C}_\delta([\xi_0,\sigma])$ and the vanishing behavior of the first-passage kernel $H$ for small time will directly yield the result.

\begin{proposition}\label{prop:fixedPoint}
For small enough $\sigma>0$, the map $\mathcal{F}_\delta$ is a contraction on $\mathcal{C}_{\delta}([\xi_0,\sigma))$ for the uniform norm denoted by $\Vert \cdot \Vert_{0,\sigma}$.
\end{proposition}

\begin{proof} We proceed in two steps: $(i)$ we justify that $\mathcal{G}_\delta$ induces a mapping $\mathcal{C}_\delta([\xi_0,\sigma)) \rightarrow \mathcal{C}_\delta([\xi_0,\sigma))$ that loses its renewal character for small enough $\sigma>0$; $(ii)$ we show that for small enough $\sigma>0$, $\mathcal{F}_\delta$ is a contraction for the uniform norm $\Vert \cdot \Vert_{0,\sigma}$, i.e., for all $\Psi_a,\Psi_b$ in $\mathcal{C}_{\delta}([\xi_0,\sigma))$, $\big \Vert \mathcal{F}_\delta[\Psi_a] - \mathcal{F}_\delta[\Psi_b] \big \Vert_{0,\sigma} \leq K \Vert \Psi_a - \Psi_b\Vert_{0,\sigma}$
with $K<1$.

$(i)$
The fact that $\mathcal{F}_\delta[\psi]$ belongs to $\mathcal{C}_{\delta}([\xi_0,\infty))$ follows from Proposition \ref{def:fixedPoint2}.
Let us  show that $\mathcal{F}_\delta[\psi]$ loses its renewal character when considered on $[0,\Psi^{-1}(\epsilon))$.
As $\delta \leq w_\Psi \leq 1/\nu$ and  $\Psi(0) =0$, we must have $ \nu \epsilon \leq \Psi^{-1}(\epsilon) \leq 1/\delta$.
Moreover, $\Psi^{-1}(\epsilon)$ is also defined as the time-changed refractory period after zero:
\begin{eqnarray*}
\Psi^{-1}(\epsilon)=  \Psi^{-1} \left(\Psi(0)+\epsilon \right) =   \gamma[\Psi](0) \, .
\end{eqnarray*}
This implies that on $[0,\Psi^{-1}(\epsilon))$, no more than one first-hitting time may occur in the inhomogeneous renewal processes determined by $\gamma[\Psi]$.
Correspondingly, over the time interval $[0, \Psi^{-1}(\epsilon)] \supset  [0,\nu \epsilon]$, the inhomogeneous renewal-type equation \eqref{eq:DuHamel} loses its renewal character to read 
\begin{eqnarray}\label{eq:nhomNotRenew}
G[\Psi](\sigma)
&=&  \int_{0}^\infty  H(\sigma,x)   q_0(x) \, \dd x + \int_0^\sigma H(\sigma-\tau,\Lambda) \, \dd G_0(\xi[\Psi](\tau)) \, . 
\end{eqnarray}
As the initial conditions prescribe $\Psi$ to coincide with $\Psi_0$ on $[\xi_0,0]$, this means that for all $0 \leq \tau \leq  \nu \epsilon$, we have
\begin{eqnarray*}
\xi[\Psi](\tau) = \Phi_0 \left( \Psi(\tau) - \epsilon \right) \, , \quad 0 \leq \tau \leq  \nu \epsilon \, ,
\end{eqnarray*}
so that remembering that $F_0 = G_0 \circ \Phi_0$ allows one to write \eqref{eq:nhomNotRenew} as
\begin{eqnarray*}
G[\Psi](\sigma)
&=&  \int_{0}^\infty  H(\sigma,x)   q_0(x) \, \dd x + \int_0^\sigma H(\sigma-\tau,\Lambda) \, \dd F_0(\Psi(\tau)-\epsilon) \, . 
\end{eqnarray*}
This shows that specifying $\mathcal{F}_\delta [\Psi]$, on $[0, \sigma)$ only requires knowledge of $\Psi$ on $[0,\sigma)$ for $\sigma \leq \nu \epsilon$.
Therefore, $\mathcal{F}_\delta$ is a mapping $\mathcal{C}_{\delta}([\xi_0,\sigma)) \rightarrow \mathcal{C}_{\delta}([\xi_0,\sigma))$ for all $0 < \sigma \leq \nu \epsilon$.


$(ii)$
Let us now consider two functions $\Psi_a$ and $\Psi_b$ in $\mathcal{C}_\delta([\xi_0,\infty))$.
For ease of notation, indexation by $a$ and $b$ will indicate throughout the proof relation to $\Psi_a$ and to $\Psi_b$, respectively.
For instance, we write $G_a=G[\Psi_a]$ and $G_b=G[\Psi_b]$.
By $(i)$, both cumulative functions $G_{a}$ and  $G_{b}$ satisfy a nonrenewal equation \eqref{eq:nhomNotRenew} with identical regular initial conditions on the interval $ [0, \nu \epsilon ]$.
As a result, the integral terms arising from $q_0$ in \eqref{eq:nhomNotRenew} are identical for both  $G_{a}$ and  $G_{b}$,
which allows one to write $\Delta G(\sigma)  = G_{a}(\sigma)-G_{b}(\sigma)$ for all $0 \leq \sigma \leq \nu \epsilon$ as
\begin{eqnarray*}
\Delta G(\sigma)
&=& 
\int_{0}^{\sigma }  H(\sigma-\tau,\Lambda) \,
 \dd \big[ F_0 (\Psi_a(\tau) -\epsilon)- F_0  (\Psi_b(\tau) -\epsilon) \big]  \, .
\end{eqnarray*}
where $F_0=G_0 \circ \Phi_0$ denotes the initial conditions for the cumulative flux of the original process $X_t$ in $\mathcal{M}([-\epsilon,0))$.
With no loss of generality, let us assume that $t_a=\Psi_a(\sigma) \leq \Psi_b(\sigma)=t_b$.
Then, performing the change of variables $t=\Psi_a(\tau)$ and $t=\Psi_b(\tau)$, we have 
\begin{eqnarray}
\int_0^\sigma  H(\sigma-\tau, \Lambda) \!\!\!\!\!\! && \dd \! \left[ F_0\big(\Psi_a(\tau)-\epsilon\big)-F_0\big(\Psi_b(\tau)-\epsilon \big) \right] \nonumber\\
&=& \int_0^{t_a} H(\sigma-\Phi_a(t), \Lambda) \, \dd F_0\big(t-\epsilon\big) - \int_0^{t_b} H(\sigma-\Phi_b(t), \Lambda) \, \dd F_0\big(t-\epsilon\big) \, ,
 \nonumber \\
&=& \int_0^{t_a} \big[ H(\sigma-\Phi_a(t), \Lambda) -H(\sigma-\Phi_b(\tau) \big] \, \dd F_0\big(t-\epsilon\big) \,  \nonumber \\
&& - \int_{t_a}^{t_b} H(\sigma-\Phi_b(t), \Lambda) \, \dd F_0\big(t-\epsilon\big) \label{eq:ineqBanach}
\end{eqnarray}
In turn, performing the change of variables $\tau=\Psi_a( t)$ in the first integral term of the equation above, denoted by $I_1$, yields
\begin{eqnarray*}
I_1(\sigma) & \leq & \int_0^{t_a} \big \vert H(\sigma-\Phi_a(t), \Lambda) - \, H(\sigma-\Phi_b(t), \Lambda) \big \vert \, \dd F_0\big(t-\epsilon\big) \, , \\
&=& \int_0^{\sigma} \big \vert H(\sigma-\tau, \Lambda) -H(\sigma-\Phi_b(\Psi_a(\tau)), \Lambda) \big \vert \, \dd F_0\big(\Psi_a(\tau)-\epsilon\big) \, , \\
&\leq&  \Vert h( \cdot , \Lambda) \Vert_{0,\sigma} \int_0^{\sigma}  \vert \tau  - \Phi_b(\Psi_a(\tau)) \vert \, \dd F_0\big(\Psi_a(\tau)-\epsilon\big)  \, , 
\end{eqnarray*}
where the last inequality follows from the fact that $H( \cdot, \Lambda)$ is $\Vert h( \cdot , \Lambda) \Vert_{0,\sigma}$-Lipschitz on $[0,\sigma]$.
Then, utilizing that $\Phi_b$ is necessarily $1/\delta$-Lipschitz for $\Psi_b$ in $\mathcal{C}_\delta([\xi_0,\infty))$, we have
\begin{eqnarray*}
I_1(\sigma) &\leq&  \Vert h( \cdot , \Lambda) \Vert_{0,\sigma} \int_0^{\sigma}  \vert \Phi_b(\Psi_b(\tau))  - \Phi_b(\Psi_a(\tau)) \vert \, \dd F_0\big(\Psi_a(\tau)-\epsilon\big)  \, , \\
&\leq&  \frac{ \Vert h( \cdot , \Lambda) \Vert_{0,\sigma}}{ \delta} \int_0^{\sigma}  \vert\Psi_b(\tau)  - \Psi_a(\tau) \vert \, \dd F_0\big(\Psi_a(\tau)-\epsilon\big)  \, , \\
&\leq& \frac{  \Vert h( \cdot , \Lambda) \Vert_{0,\sigma}}{ \delta} \left[ F_0(t_a-\epsilon\big) - F_0\big(-\epsilon) \right] \Vert \Psi_a -\Psi_b \Vert_{0,\sigma}
\end{eqnarray*}
where the last inequality follows from the fact $F_0$ is an increasing function.
The last integral term in  \eqref{eq:ineqBanach}, denoted by $I_2$, can be bounded via similar argument as
\begin{eqnarray*}
I_2(\sigma) &=&  \int_{t_a}^{t_b} H(\sigma-\Phi_b(t), \Lambda)   \, \dd F\big(t-\epsilon\big) \nonumber\\
 &\leq& 
 H(\sigma-\Phi_b(t_a), \Lambda) \, \left[ F(t_b-\epsilon\big) - F\big(t_a-\epsilon) , \Lambda) \right] \, , \\
&\leq& 
\Vert h( \cdot , \Lambda) \Vert_{0,\sigma} \vert  \Phi_b(\Psi_b(\sigma)) - \Phi_b(\Psi_a(\sigma)) \vert  \, \left[ F(t_b-\epsilon\big) - F\big(t_a-\epsilon) \right] \, , \\
 &=& 
\frac{ \Vert h( \cdot , \Lambda) \Vert_{0,\sigma}}{\delta}  \left[ F(t_b-\epsilon\big) - F\big(t_a-\epsilon) \right] \, \Vert \Psi_a -\Psi_b \Vert_{0,\sigma} \, ,
\end{eqnarray*}
Since by conservation of probability $F(t_b-\epsilon\big) - F\big(-\epsilon) \leq 1$, we have for all $\sigma \geq 0$
\begin{eqnarray*}
\vert \Delta G(\sigma) \vert
= \frac{\lambda}{\nu}  \big \vert I_1(\sigma) + I_2(\sigma) \big \vert
\leq 
\frac{ \lambda \Vert h( \cdot , \Lambda) \Vert_{0,\sigma}}{ \nu \delta} \Vert \Psi_a -\Psi_b \Vert_{0,\sigma} \, .
\end{eqnarray*}
Moreover, for all real valued functions $\psi_a$ and $\psi_b$ over $\mathbbm{R}^+$, we have
\begin{eqnarray*}
\big \vert \mathcal{S}_\delta(\psi_a) - \mathcal{S}_\delta(\psi_b) \big \vert 
&=& 
\Big \vert \sup_{0 \leq s \leq \sigma} \{ \psi_a(s) - \delta s \} - \sup_{0 \leq s \leq \sigma} \{ \psi_b(s) - \delta s \} \Big \vert \, , \\
&\leq& 
\Big \vert \sup_{0 \leq s \leq \sigma} \{ \psi_a(s) -  \psi_b(s)  \} \Big \vert \, , \\
&\leq& 
\sup_{0 \leq s \leq \sigma} \big \vert  \psi_a(s) -  \psi_b(s)   \big \vert = \Vert \psi_a -  \psi_b \Vert_{0,\sigma}\, ,
\end{eqnarray*}
so that we have the inequality
\begin{eqnarray*}
\big \vert \mathcal{F}_\delta[\Psi_b](\sigma)-\mathcal{F}_\delta[\Psi_a](\sigma) \big \vert
&=&
\big \vert \mathcal{S}_\delta[(\mathrm{id}-\lambda G_b)/\nu](\sigma)-\mathcal{S}_\delta[(\mathrm{id}-\lambda G_a)/\nu](\sigma) \big \vert \, , \\
& \leq &
\big \Vert \lambda G_b/\nu-\lambda G_a/\nu] \big \Vert_{0,\sigma}  \, , \\
&\leq &
\frac{ \lambda \Vert h( \cdot , \Lambda) \Vert_{0,\sigma}}{ \nu \delta}  \Vert \psi_a- \psi_b  \Vert_{0,\sigma} \, .
\end{eqnarray*}
We conclude by noticing that $\lim_{\sigma \to 0}  h(\sigma, \Lambda) = 0$, so that for small enough $\sigma>0$, $K= \lambda   \, h(\sigma, \Lambda)/ (\nu \delta)<1$, which shows that the map $\mathcal{F}_\delta$ is a contraction on the space $\mathcal{C}_\delta(0,\sigma)$ for the uniform norm $\Vert \cdot \Vert_{0,\sigma}$.
\end{proof}

By the Banach fixed-point theorem, Proposition \ref{prop:fixedPoint} implies the existence of a unique local solution to the regularized fixed-point problem of Proposition \ref{def:fixedPoint2}.
The following corollary shows that for all $\delta>0$, such a local solution can be maximally extended to the whole real half-line $\mathbbm{R}^+$.

\begin{corollary}\label{cor:globSol}
For all finite refractory periods $\epsilon>0$ and for all parameter $0<\delta<1/\nu$, there is a unique global solution $\psi_\delta$ in $\mathcal{C}_\delta([\xi_0,\infty))$ to the fixed-point problem $\psi=\mathcal{G}_\delta[\psi]$. 
\end{corollary}

\begin{proof} 
For fixed $\delta>0$, a local solution can be continued unconditionally as $\mathcal{F}_\delta$ is locally contracting irrespective of the initial conditions.
More specifically, assuming the solution defined up to $\tau>0$, the continuation process past $\tau$ consists in applying Proposition \ref{prop:fixedPoint} with $\{ g(\sigma) \}_{\tau \leq \sigma \leq \xi(\tau)}$ serving as initial flux conditions,
and with spatial part of the initial condition naturally given by
\begin{eqnarray*}
q(\tau,y) = \kappa(\tau,y,x) \, p_0(x) \, dx + \int_0^\tau \kappa(\tau-\sigma,y,\Lambda) \, \dd G(\xi(\sigma)) \, .
\end{eqnarray*}
\end{proof}

We are now in a position to exhibit a local solution to the dPMF fixed-point problem \ref{def:fixedPoint} by considering the function $\Psi=\lim_{\delta \to 0^+} \Psi_\delta$.
For this solution to be uniquely defined on a nonzero interval $[0,S_1)$, we require that the initial conditions are such that instantaneous blowups are excluded.
Specifically, we assume that the density $q_0$ is locally smooth near zero with $q_0(0)=0$ and $\partial_x q_0(0) /2<1/\lambda$.
This amounts to imposing that all solutions $\Psi_\delta$ are such that ${\Psi'_0}(0^+) = 1-\lambda \partial_x q_0(0) /2$ is bounded away from zero, so that for all $0<\delta_1<\delta_2<\Psi'_0(0^+)$, $\Psi_{\delta_1}(\sigma)=\Psi_{\delta_2}(\sigma)$ for small enough $\sigma>0$.
In turn, such a local solution $\Psi$ can be uniquely continued up to the first time $\Psi'$ becomes zero, therefore giving a criterion to maximally define $S_1$.
This leads to the following proposition:

\begin{theorem}\label{th:smooth}
For all normalized initial conditions $(q_0,g_0)$ in $\mathcal{M}(\mathbbm{R}^+) \times\mathcal{M}([\xi_0,\infty))$ such that $q_0(0)=0$ and  $\partial_x q_0(0) /2<1/\lambda$, there is a unique smooth solution to the fixed-point problem $\Psi=\mathcal{F}_0[\Psi]$ up to the possibly infinite time 
\begin{eqnarray}\label{eq:defS0}
S_1
=
\inf \left\{ \sigma>0 \, \big \vert \, \Psi'(\sigma)  \leq 0 \right\}   
=
\inf \left\{ \sigma>0 \, \big \vert \, g[\Psi](\sigma)  \geq 1/\lambda \right\}  > 0 \, .
\end{eqnarray}
\end{theorem}

\begin{proof}
Suppose there exists a solution $\Psi$ to the fixed-point problem $\Psi=\mathcal{F}_0[\Psi]$.
Then, the smooth function $\psi=( \mathrm{id} - \lambda G[\Psi]  ) / \nu $ in $\mathcal{C}_0([\xi_0,\infty))$ is such that $\Psi=\mathcal{S}_0[\psi]$ with $\psi'_0(0^+)=\Psi'_0(0^+)=(1-\lambda \partial_x q_0(0) /2) /\nu>0$.
Let us then introduce the positive time 
\begin{eqnarray*}
S_{1,\delta}=\inf \left\{ \sigma >0 \, \big \vert \, \psi'(\sigma) \leq \delta \right\} >0\, ,
\end{eqnarray*}
so that $\Psi=\psi$ is necessarily smooth on $[0,S_{1,\delta})$.
For all $\delta>0$, $\Psi$ is also determined as the solution of the regularized fixed-point problem $\Psi=\mathcal{F}_\delta[\Psi]$ on the possibly infinite time interval $[0,S_{1,\delta})$.
Indeed, on $[0,S_{1,\delta})$ we have  
\begin{eqnarray*}
\mathcal{F}_\delta [\Psi](\sigma) 
= \mathcal{S}_\delta \left[\big( \sigma - \lambda G[\Psi](\sigma)  \big) / \nu \right]
= \mathcal{S}_\delta \left[\psi \right]
= \psi
=  \Psi  \, , \quad 0 \leq \sigma \leq S_{1,\delta}\, .
\end{eqnarray*}
By Corollary \ref{cor:globSol}, there is a unique solution $\Psi_\delta$ to $\Psi=\mathcal{F}_\delta[\Psi]$ in $\mathcal{C}_\delta([\xi_0,\infty)) \subset \mathcal{C}_0([\xi_0,\infty))$.
Thus, the time $S_{1,\delta}$ is equivalently defined as
\begin{eqnarray*}
S_{1,\delta}=\inf \left\{ \sigma >0 \, \big \vert \, \psi'_\delta(\sigma) \geq \delta \right\} >0\, , \quad \mathrm{with} \quad  \psi_\delta = \big( \sigma - \lambda G[\Psi_\delta](\sigma)  \big) / \nu \, ,
\end{eqnarray*}
and $\Psi=\mathcal{F}_0[\Psi]$ admits $\Psi_\delta$ as unique smooth solution on $[0,S_\delta)$.
Moreover, for all $\delta_2>\delta_1>0$, since we have $\Psi_{\delta_1}=\Psi_{\delta_2}$ on $[0,S_{\delta_2})$ and since $S_{1,\delta}$ is a decreasing function of $\delta$, the solution to $\Psi=\mathcal{F}_0[\Psi]$ can be uniquely extended to a smooth function on 
$[0,S_1)$, with $S_1 = \lim_{\delta \to 0} S_{1,\delta}$.
Finally, it remains to check that $S_1$ is also defined as \eqref{eq:defS0}.
For all $0 \leq \sigma < S_1$, there is $\delta>0$ such that $S_{1,\delta}>\sigma$, so that we have $\psi'(\sigma)=\psi'_\delta(\sigma) > \psi'_\delta(S_{1,\delta}) = \delta>0$.
There is nothing more to show if  $S_1=\infty$.
If $S_1<\infty$, as a bounded increasing function on $[0,S_1)$, the solution $\psi$ admits a left limit in $S_1$.
Thus $\psi$ can be extended by continuity to $[0,S_1]$ with: 
\begin{eqnarray*}
\psi'(S_1) = \lim_{\delta \to 0^+}\psi'(S_{1,\delta})  =\lim_{\delta \to 0^+} \psi'_\delta(S_{1,\delta}) = \lim_{\delta \to 0^+} \delta = 0 \, .
\end{eqnarray*}
\end{proof}

\subsection{Analytical characterization of full blowups}

Assuming the dynamics starting at time zero to be initially smooth with $q_0(0)=0$ and $\partial_x q_0(0)/2<1/\lambda$, the first blowup time $T_1$ occurs when the cumulative flux $f$, or equivalently $\Phi'=1+ \lambda f$, first diverges, i.e.,
\begin{eqnarray*}
T_1
&=&
\inf \left\{ T>0 \, \Big \vert \, \lim_{t \to T^-}  \Phi'(t)  = \infty \right\}   \, ,\\
&=&
\inf \left\{ \Psi(S)>0 \, \Big \vert \,  \lim_{\sigma \to S^-}  \Psi'(\sigma) = 0 \right\}  = \Psi(S_1) > 0 \, .
\end{eqnarray*}
Thus, in the time-changed picture, the  blowup condition corresponds precisely to the definition of $S_1$, the terminal point of the interval over which Theorem \ref{th:smooth} guarantees the existence of smooth, increasing solutions.
This justifies defining the following blowup conditions with respect to the time-changed dynamics:

\begin{definition}\label{def:blowup}
A blowup occurs if $T_1 =  \Psi(S_1)<\infty$ which is equivalent to 
\begin{eqnarray*}
S_1 = \inf \big\{ \sigma>0 \, \big \vert \, g(\sigma) \geq 1/\lambda  \big\} < \infty \, ,
\end{eqnarray*}
 where $g = \partial_\sigma G$ is the  instantaneous inactivation flux for the time-changed process $Y_\sigma$.
\end{definition}

As stated at the end of Section \ref{subsec:MKV}, the emergence of blowup clearly depends on the initial conditions.
However, for large enough interaction parameter, blowups will generically occur in finite time.
To see this, consider for instance initial conditions of the form $(q_0,g_0)=(\delta_{x_0},0)$.
For such initial conditions, the fixed-point problem $\Psi=\mathcal{F}_\delta[\Psi]$ admits an initially smooth solution $\Psi$ with instantaneous flux $g$ satisfying
\begin{eqnarray*}
g(\sigma) = h(\sigma, x_0) + \int_0^\sigma   h(\sigma-\tau, \Lambda) \, \dd G(\xi(\tau)) \geq h(\sigma, x) \, ,
\end{eqnarray*}
where $h(\sigma,x)= \partial_\sigma H(\sigma, x) \geq 0$ represents the first-passage density to zero of a Wiener process started in $x$ with negative unit drift:
\begin{eqnarray}\label{eq:kFPT}
h(\sigma,x) = \frac{x}{\sqrt{2 \pi \sigma^3}} \exp \left(-\frac{(x-\sigma)^2}{2 \sigma} \right)\, .
\end{eqnarray}
Moreover, we have $\sup_{\sigma \geq 0} h(\sigma, x_0) \geq   h(x_0, x_0) = 1 /\sqrt{2 \pi x_0}$.
This implies that the blowup condition $g(\sigma) \leq 1/\lambda$ is satisfied in finite time whenever $\lambda> C_{x_0}=\sqrt{2 \pi x_0}$.

From now on, let us consider that the blowup condition is first met in $S_1<\infty$.
On $[0,S_1]$, the inverse time change $\Psi$ is a smooth function with $\lim_{\sigma \to S_1^-} \Psi'(\sigma) = 0$.
Thus,  the diverging behavior of $f = (\Phi'-\nu)/\lambda$ is determined by the order of the first nonzero left-derivative of $\Psi$ in $S_1$, which is always larger or equal to two.
In all generality, this order depends on the initial conditions.
However, for generic initial conditions, we expect that
\begin{eqnarray*}
\inf \left\{ n \geq 1 \, \Big \vert \, \lim_{\sigma \to S_1^-} \Psi^{(n)}(\sigma) \neq 0 \right\} = 2 \, .
\end{eqnarray*}
Moreover, given that we necessarily have $\Psi'>0$ on $[0,S_1)$, $S_1$ must be a local maximum so that the criterion $\lim_{\sigma \to S_1^-} \Psi^{(n)}(\sigma) \neq 0$ is actually equivalent to $\lim_{\sigma \to S_1^-} \Psi^{(n)}(\sigma) > 0$.
The above observations lead us to introduce an additional condition for blowups, which we refer to as the full-blowup condition:

\begin{definition}\label{def:strictblowup}
The blowup time $T_1 = \Phi(S_1)$ satisfies the full-blowup condition if $\lim_{\sigma \to S_1^-}\partial_\sigma g(\sigma)>0$.
\end{definition}

The definition of the full-bowup condition naturally follows from the fact that $\Psi'=(1-\lambda g) / \nu$ on $[0,S_1]$, so that $\lim_{\sigma \to S_1^-} \Psi''(\sigma)=\lim_{\sigma \to S_1^-} \partial_\sigma g(\sigma)$.
It is straightforward to check that full blowups are marked by H\"older singularity with exponent $1/2$ for the time change $\Phi$:

\begin{proposition}
Under the full-blowup condition, the flux density $f$ diverges in $T_1$ as 
\begin{eqnarray*}
f(t) \stackrel[t \to T_1^-]{}{\sim} \frac{ \lambda}{
 \sqrt{2a_1(T_1-t)}} \, , \quad \mathrm{with} \quad a_1 = (\lambda/\nu) \partial_\sigma g(S_1) \, .
\end{eqnarray*}
\end{proposition}

\begin{proof}
The generic blowup condition implies that the inverse time change $\Psi$ admits a zero left derivative when $t \to T_1^-$.
The full-blowup condition further ensures that $\Psi$ behaves locally quadratically in the left vicinity of $S_1$.
Specifically, we have
\begin{eqnarray*}
\Psi(\sigma) = T_1 - a_1(\sigma-S_1)^2/2+o\big((\sigma-S_1)^2\big) \, , \quad  \sigma < S_1 \, ,
\end{eqnarray*}
where the quadratic coefficient is given by
\begin{eqnarray*}
a_1=-\partial^2_\sigma \lim_{\sigma \to S_1}\Psi(\sigma) = (\lambda/\nu) \partial_\sigma g(S_1)  >0 \, .
\end{eqnarray*}
Thus, for $t<T_1$,  just before blowup,  the time change $\Phi$ behaves as
\begin{eqnarray*}
\Phi(t) = S_1 - \sqrt{2(T_1-t)/a_1} + o\big(\sqrt{T_1-t}\big) \, .
\end{eqnarray*}
In turn, this implies a blowup divergence as the reciprocal of a square root: 
\begin{eqnarray*}
f(t) 
=
\frac{\nu g \left(\Phi (t)\right)}{1 - \lambda g \left(\Phi (t)\right)}
\stackrel[t \to T_1^-]{}{\sim}
\frac{\nu \lambda}{\lambda \partial_\sigma g(S_1)\big(S_1-\Phi(t)\big)}
=
\frac{ \lambda}{
 \sqrt{2a_1(T_1-t)}} \, .
\end{eqnarray*}
\end{proof}

A synchronization event occurs in $T_1$ if the time change $\Phi$ exhibit a jump discontinuity in $T_1$ after a blowup.
Such a discontinuity corresponds to the inverse time change $\Psi$ being flat on a non-empty interval $[S_1,U_1)$, with $S_1=\Phi(T_1^-)<U_1=\Phi(T_1)$.
By smoothness of $\Psi$ on $[0,S_1]$, every synchronization event is triggered by a blowup but in all generality, a blowup need not trigger a synchronization event, which corresponds to the marginal case $S_1=\Phi(T_1^-)=U_1=\Phi(T_1)$.
However, under the full-blowup condition, a blowup always trigger a synchronization event, i.e., $S_1=\Phi(T_1^-)<U_1=\Phi(T_1)$.

\begin{theorem}\label{th:jump}
Suppose $S_1$ is a full blowup time for the time-changed dynamics with normalized initial conditions $(q_0,g_0)$ such that $q_0(0)=0$ and $\partial_x q_0(0)/2<1/\lambda$.
Then, the solution to the fixed-point problem $\Psi=\mathcal{F}_0[\Psi]$ can be uniquely extended as a constant function on $[S_1,U_1]$ with $U_1=S_1+\lambda \pi_1$ where $\pi_1$ satisfies $0<\pi_1<\Vert q(S_1, \cdot) \Vert_1 \leq 1$ and is defined as 
\begin{eqnarray}\label{eq:defjump}
\pi_1 = \inf \left\{ p \geq 0 \, \bigg \vert \, p > \int_{0}^\infty H(\lambda p, x )  q(S_1,x ) \, \dd x \right\} \, .
\end{eqnarray}
\end{theorem}

\begin{proof}
The proof proceeds in two steps:
$(i)$ We show that if it is possible to locally extend a solution $\Psi$ past a full blowup in $S_1$, such an extension is uniquely determined on the maximum interval $[S_1,U_1]$, where $\Psi$ is constant.
$(ii)$ We show that under full-blowup conditions, it is always possible to extend a solution past $S_1$, i.e., $S_1<U_1$.

$(i)$ Consider the smooth solution $\Psi$ of the fixed-point problem $\Psi=\mathcal{F}_\delta[\Psi]$ on $[0,S_1]$.
Suppose there exists $U>S_1$ such that the solution $\Psi$ can be extended on $[0,U]$.
Then on $[0,U]$, there is a smooth function $\psi$ such that $\Psi=\mathcal{S}_0[\psi]$ with $\psi=(\mathrm{id}-\lambda G[\Psi])/\nu$.
Moreover, the full-blowup condition entails that  $ \psi'(S_1) = \lim_{\sigma \to S_1^-} \Psi(\sigma)=0$ and $\psi''(S_1) =  \lim_{\sigma \to S_1^-}\Psi''(\sigma)<0$.
Thus, $\psi$ must be locally decreasing for $\sigma>S_1$, so that 
\begin{eqnarray*}
U^\dagger = \sup \left\{\sigma \in [0,U] \, \big \vert \, \psi(\sigma) \leq \psi(S_1)=\Psi(S_1) \right\} >0 \, .
\end{eqnarray*}
Then, for all $0 \leq \sigma \leq U^\dagger$, we have $\psi(\sigma) \leq \Psi(S_1)$ so that $\Psi=\mathcal{S}_0[\psi]$ implies that $\Psi(\sigma)=\Psi(S_1)$ on $[S_1,U^\dagger]$. 
This shows that under the full blowup condition, smooth solution $\Psi$ can only be continued locally as a constant function, if it is possible at all.
Suppose that $\Psi$ is constant for all $0 \leq \sigma \leq U_1$, for some real value $U_1>S_1$.
Then the backward function $\xi(\sigma)=\sigma-\eta(\sigma)=\Phi(\Psi(\sigma)-\epsilon)=\Phi(-\epsilon)$ is also constant.
Consequently, the integral term in the renewal-type equation \eqref{eq:DuHamel} attached to the fixed-point problem \ref {def:fixedPoint} vanishes on $[S_1, U_1]$ and  for all $S_1 \leq \sigma \leq U_1$, we must have $\Psi=\mathcal{S}_0[\psi]$ where: 
\begin{eqnarray*}
 \psi(\sigma)=\big(\sigma-\lambda G(\sigma) \big)/\nu \quad \mathrm{with} \quad G(\sigma) =  G(S_1) + \int_0^\infty H(\sigma,x) q(S_1,x) \, dx  \, .
\end{eqnarray*}
Thus the auxiliary function $\psi$ is uniquely determined on $[0,U_1]$.
For this determination to be consistent, we must have that for all $0 \leq \sigma \leq U_1$, $\psi(\sigma) \leq \Psi(S_1)$, which is equivalent to
\begin{eqnarray*}
\nu\big(\psi(\sigma)- \Psi(S_1) \big)
&=&
\sigma-S_1-\lambda \big(G(\sigma) - G(S_1) \big) \, , \\
&=&
\sigma-S_1 + \lambda \int_0^\infty H(\sigma,x) q(S_1,x) \, dx \leq 0 \, . \\
\end{eqnarray*}
This shows that the maximum possible value $U_1$ ensuring that $\Psi$ is constant on $[0,U_1]$ is:
\begin{eqnarray*}
\quad \quad U_1&=& \inf \left\{\sigma>0 \, \big \vert \, \psi(\sigma) > 0 \right\}  = \lambda  \inf \left\{p>0 \, \bigg \vert \, p - \int_0^\infty H(\lambda p,x) q_0(x) \, dx  > 0 \right\}  
= \lambda \pi_1 \, .
\end{eqnarray*}

$(ii)$ To conclude, it remains to show that the interval $S_1<U_1$ is nonempty under full-blowup condition, so that $\Psi$ can be maximally continued as a constant function on a nonempty interval $[S_1,U_1)$.
To that end, we will actually show that the number $\pi_1$, as defined  in \eqref{eq:defjump}, is such that $0<\pi_1< \Vert q_0 \Vert_1$ under full-blowup condition.. 
Let us consider the smooth function $\zeta$ defined on $[0,\infty)$ by
\begin{eqnarray*}
\zeta(p) = p- \int_{0}^\infty  H(\lambda p, x)  q(S_1,x) \, \dd x   \, .
\end{eqnarray*}
Our first goal is to prove that  $\pi_1$ exists as a root of the function $\zeta$ and is such that $0<\pi_1 \leq \Vert q_0 \Vert_1 \leq 1$.
To show this, observe that $\zeta$ is a smooth function on $[0,\infty)$ such that for all $p>\Vert q(S_1, \cdot)\Vert_1$, we have
\begin{eqnarray*}
\zeta(p) \geq p - \Vert H(\lambda p, \cdot ) \Vert_\infty \int_0^\infty q(S_1, x)  \, \dd x = p - \Vert q(S_1, \cdot)\Vert_1 > 0 \, ,
\end{eqnarray*}
Thus to establish the existence of a root  in $(0,\Vert q(S_1, \cdot)\Vert_1)$, it is enough to show that $\zeta$ takes negative values in $(0,\Vert q(S_1, \cdot)\Vert_1)$.
In fact, we will show that $\zeta$ is negative in the vicinity of zero under full blowup condition.
As $\zeta$ satisfies $\lim_{p \to 0^+} \zeta(p)=0$, we first consider the asymptotic behavior of its derivative function given by 
\begin{eqnarray*}
\zeta'(p) = 1 -  \lambda \int_{0}^\infty h(\lambda p, x)  q(S_1, x) \, \dd x \, ,
\end{eqnarray*}
where $h$ is defined as the first-passage density in \eqref{eq:kFPT}.
The limit behavior of $\zeta'$ is
\begin{eqnarray*}
\lim_{p \to 0^+} \zeta'(p) = 1 - \lambda \lim_{p \to 0^+} \int_{0}^\infty h(\lambda p, x)  q(S_1,x) \, \dd x = 1 - \lambda \partial_x q(S_1,0)/2  \, ,
\end{eqnarray*}
where the last equality follows from the absorbing boundary condition in zero  and the asymptotic property of first-passage density $h$: $\lim_{t \to 0} h(t, x) = \delta(x)-\delta'(x)/2$ in the sense of generalized distributions given by \ref{app:Asympt1}.
From there, under blowup condition, we have  
\begin{eqnarray*}
\lim_{p \to 0} \zeta'(p) = 1 - \lambda \partial_x q(S_1,0)/2 = 1 - \lambda g(S_1,0) = 0 \, .
\end{eqnarray*}
Next, we evaluate the limit of the second derivative $\zeta''$.
To do so, we utilize the asymptotic result obtained in \ref{app:Asympt2}.
To apply this result, we utilize the facts that $(1)$ $x \mapsto q(S_1,x)$ admits a locally bounded fourth derivative in zero with $q(0,x)=0$ 
and that $(2)$ injecting $q(\sigma,0)=0$ for all $\sigma>0$ in \eqref{eq:qPDE}  yields $\partial_x q(S_1,0)+\partial^2_x q(S_1,0)/2=0$.
The asymptotic result from \ref{app:Asympt2} implies that
\begin{eqnarray*}
\lim_{p \to 0^+} \zeta''(p) 
&=&  - \lambda^2 \lim_{p \to 0^+} \int_{0}^\infty \partial_\sigma h(\lambda p, x)  q(S_1,x) \, \dd x \, , \\
&=&  -\frac{\lambda^2}{2} \left( \partial^2_x q(S_1, 0) + \partial^3_x q(S_1, 0)/2 \right)  \, .
\end{eqnarray*}
Differentiating \eqref{eq:qPDE} with respect to $x$ below the reset site, i.e., for $x <\Lambda$, we get
\begin{eqnarray*}
\partial_x \partial_\sigma q =  \partial^2_x q+ \partial^3_x q/2 \, ,
\end{eqnarray*}
so that specifying the above relation in $(S_1, 0)$ yields:
\begin{eqnarray*}
\lim_{p \to 0^+} \zeta''(p) 
=  
-\lambda^2 \partial_x \partial_\sigma q (S_1, 0) / 2 \, .
\end{eqnarray*}
Finally, permuting  the order of the partial derivatives in the cross-derivative term yields  $\partial_\sigma \partial_x q(S_1,0)/2 = \partial_\sigma g(S_1)$ so that, under the full blowup condition, we get
\begin{eqnarray*}
\lim_{p \to 0} \zeta''(p) =- \lambda^2 \partial_\sigma g(S_1)  < 0 \, .
\end{eqnarray*}
This implies that as a root of $\zeta$, $\pi_1$ exists and is such that $0<\pi_1 \leq \Vert q_0 \Vert_1 \leq 1$.
\end{proof}

The above proposition has a direct interpretation in terms of the original dPMF dynamics.
In the event of a full blowup, the dynamics of the time-changed process $Y$ cannot unfold smoothly after the blowup time $S_1$ as it would imply that $\Psi'=(1-\lambda g)/\nu<0$ on some nonempty interval to the left of $S_1$. 
In other word, as a decreasing function, the function $\Psi$ would implement a time-reversal in $S_0$, which is not physically admissible.
Physical solutions resolve this conundrum by freezing the clock for the original time at $T_0=\Psi(S_0)$, while letting the clock for the changed time run past $S_0$.
In the time-changed picture, this corresponds to stalling the reset of inactive processes, while letting active processes inactivate according to their linear, noninteracting dynamics.
Such a non-reset dynamics continues in the time-changed picture until the original clock can start running again, which happens at time $U_1=S_1+\lambda \pi_1$.
Incidentally, the number $0<\pi_1<1$ is the fraction of processes that synchronously inactivates at time $T_0$, which is marked by a discontinuity of size $\lambda \pi_1$ in the time change $\Phi=\Psi^{-1}$. 
We summarize the above discussion by stating the following corollary.

\begin{corollary}\label{cor:jump}
Under the full-blowup condition at time $T_1=\Psi(S_1)$, a synchronization event occurs with size $0<\big(\Phi(T_1)-\Phi(T_1^-)\big)/\lambda=\pi_1<1$.
\end{corollary}

Observe that the definitions of the generic blowup trigger time $S_1$ and of the blowup exit time $U_1=S_1+\lambda \pi_1$ is rather imprecise with respect to the behavior of $\Psi$ in the immediate vicinity of $S_1$ and $U_1$.
These imprecisions are the sources of difficulties in extending the existence of a solution over the whole real half-line $\mathbbm{R}^+$.
To exhibit such a solution using our prior results, we need to check that the auxiliary function $\psi= (\mathrm{id}-\lambda G )/\nu$ is such that $\psi'(U_1)>0$ at blowup exit time $U_1$, so that we can invoke Theorem \ref{th:smooth} to continue the solution over some nonempty interval $[U_1,S_2)$, where $S_2$ is the next putative blowup time where $\lim_{\sigma \to S_2-} \Psi'(\sigma)=0$.
Then, if $S_2<\infty$, invoking Theorem \ref{th:jump} to further continue the solution via blowup resolution necessitates checking the full-blowup condition: $\lim_{\sigma \to S_2-} \Psi''(\sigma)<0$.
Assuming that all these conditions check ad infinitum, exhibiting a solution over the whole real half-line $\mathbbm{R}^+$ will finally require to exclude the occurrence of accumulation points, whereby an infinite number of vanishingly small blowups happens in finite time.
The main result of \cite{TTLS} is to show that all these checks and requirements are met for sufficiently large interaction parameter $\lambda$ and sufficiently small refractory period $\epsilon>0$.
Establishing this result relies on a detailed analysis of the time-changed dynamics and is beyond the scope of this work, which is mainly concerned with introducing the time-changed picture to characterize mean-field dynamics with blowups.

\section*{Acknowledgements}
The authors would like to thank the anonymous referees, the Associate
Editor, and the Editor for their constructive comments that improved the
quality of this paper.

The first author was supported by an Alfred P. Sloan Research Fellowship FG-2017-9554 and a CRCNS award DMS-2113213 from the National Science Foundation.

The second author was supported in part by a grant from the Center for Theoretical and Computation Neuroscience from the University of Texas, Austin.

%
%
%

\begin{appendix}

\section{Asymptotic behavior of $h(\sigma, \cdot)$ and $\partial_\sigma h(\sigma, \cdot)$ when $\sigma \to 0^+$}\label{appA}

This appendix comprises two useful results about the short-time asymptotics of the first-passage kernel $(\sigma,x) \mapsto h(\sigma, x)$ and its time-derivative $(\sigma,x) \mapsto \partial_\sigma h(\sigma, x)$.
The first result follows from classical work in ~\cite{Roz84}, whereas the second result requires original analysis.

\begin{proposition}\label{app:Asympt1}
Consider a continuous function $q:\mathbbm{R} \to \mathbbm{R}$.
Suppose moreover that $q$ is continuously differentiable on $[0,\delta]$ for some $\delta>0$, then
\begin{eqnarray}
\lim_{\sigma \to 0^+} \int_0^\infty \partial_\sigma h(\sigma,x) q(x) \, \dd x 
&=&
 q(0) + \partial_x q(0)/2 \, .
\end{eqnarray}
\end{proposition}

\begin{proof}
 After a simple change of variable $y=x-t$, one can check that
\begin{eqnarray}
\int_{0}^\infty h(\sigma,x) q(x) \, dx
= \int_{0}^\infty \frac{ye^{-\frac{y^2}{2\sigma}}}{\sqrt{2 \pi \sigma^3}} f(y+\sigma)  \, dy
+\int_{0}^\infty \frac{e^{-\frac{(x-\sigma)^2}{2\sigma}}}{\sqrt{2 \pi \sigma}} f(x)  \, dx
\end{eqnarray}
where we recognize the drifted heat kernel in the last integral term.
The asymptotic behavior of the first integral term follows the analysis in~\cite{Roz84}, which shows that
\begin{eqnarray}\label{eq:heatAsympt}
\lim_{\sigma \to 0+} \int_{0}^\infty \frac{xe^{-\frac{x^2}{2\sigma}}}{\sqrt{2 \pi \sigma^3}} q(x) \, dx = \partial_x q(0)/2 \, ,
\end{eqnarray}
for all functions $q$ with locally bounded derivative in zero.
Thus, we have
\begin{eqnarray}\label{eq:h1Asympt}
\lim_{\sigma \to 0+}  h(\sigma,x) = \delta_0(x) -  \delta'_0(x)/2 \, .
\end{eqnarray}
in the distribution sense.
\end{proof}

\begin{proposition}\label{app:Asympt2}
Consider a continuous function $q:\mathbbm{R} \to \mathbbm{R}$ with nonnegative value and polynomial growth on $\mathbbm{R}^+$.
Suppose moreover that $q$ is four times continuously differentiable on $[0,\delta]$ for some $\delta>0$ and $q(0)=0$ and $\partial_x q(0)+\partial^2_x q(0)/2=0$, then
\begin{eqnarray}
\lim_{\sigma \to 0^+} \int_0^\infty \partial_\sigma h(\sigma,x) q(x) \, \dd x 
&=&
\frac{1}{2} \left( \partial^2_x q(0) + \partial^3_x q(0)/2\right)\, .
\end{eqnarray}
\end{proposition}

\begin{proof}
As $q:\mathbbm{R} \to \mathbbm{R}$  has polynomial growth, i.e., there is an integer $d>0$ such that $q(x) \leq K (1+x^d)$ on $\mathbbm{R}^+$ for some real $K>0$.
For all real $\delta>0$ and all integers $m, n \geq 0$, we have
\begin{eqnarray}
\lim_{\sigma \to 0^+} \frac{1}{\sigma^m}\int_\delta^\infty e^{-\frac{(x-\sigma)^2}{2\sigma}}x^n q(x) \, \dd x = 0 \, .
\end{eqnarray}
This follows from the fact that $x \mapsto e^{-\frac{(x-\sigma)^2}{2\sigma}}x^n$ is decreasing for $x>\left(\sigma+\sqrt{\sigma(4n+\sigma)}\right)$. Then for $0 \leq \sigma<\delta^2/(n+\delta)$, we have:
\begin{eqnarray}
\sup_{x\geq \delta} e^{-\frac{(x-\sigma)^2}{2\sigma}}x^n = e^{-\frac{(\delta-\sigma)^2}{2\sigma}} \delta^n \, .
\end{eqnarray}
This allows one to write for $0 \leq \sigma<\delta^2/(n+d+2+\delta)$
\begin{eqnarray}
\int_\delta^\infty e^{-\frac{(x-\sigma)^2}{2\sigma}}x^n q(x) \, \dd x
&\leq&
\int_\delta^\infty e^{-\frac{(x-\sigma)^2}{2\sigma}}x^n\big(1+x^{d+2} \big) \frac{K(1+x^d)}{1+x^{d+2}} \, \dd x \, , \\
&\leq&
\, e^{-\frac{(\delta-\sigma)^2}{2\sigma}}\delta^n K \big(1+\delta^{d+2} \big)   \int_\delta^\infty  \frac{1+x^d}{1+x^{d+2}} \, \dd x \, , \\
&\leq&
K_{d,\delta} \, \delta^n  e^{-\frac{(\delta-\sigma)^2}{2\sigma}} \, , 
\end{eqnarray}
where the constant $K_{d,\delta}$ only depends on  $d$ via the Gamma function:
\begin{eqnarray}
K_{d,\delta} 
&=& K \big(1+\delta^{d+2} \big) \int_0^\infty  \frac{1+x^d}{1+x^{d+2}} \, \dd x \, , \\
&=& K \big(1+\delta^{d+2} \big) \left(2 \Gamma \left[\frac{1+d}{2+d} \right] \Gamma \left[\frac{3+d}{2+d} \right] \right) < \infty \, .
\end{eqnarray}
We conclude by observing that
\begin{eqnarray}
0
\leq
\frac{1}{\sigma^m}\int_\epsilon^\infty e^{-\frac{(x-\sigma)^2}{2\sigma}}x^n q(x) \, \dd x
\leq
K_{d,\delta} \, \delta^n e^{-\frac{(\delta-\sigma)^2}{2\sigma}}/\sigma^m   \xrightarrow{\sigma \to 0^+} 0
\end{eqnarray}
The above observation implies that if $q$ is a function with polynomial growth, then for all $\delta>0$ we have
\begin{eqnarray}\label{eq:remainBound}
\lim_{\sigma \to 0^+} \int_\delta^\infty  \vert \partial_\sigma h(\sigma,x) \vert q(x) \, \dd x = 0 \, ,
\end{eqnarray}
so that if the limits at stake exist, we have
\begin{eqnarray}
\lim_{\sigma \to 0^+} \int_0^\infty \partial_\sigma h(\sigma,x) q(x) \, \dd x 
&=&
\lim_{\sigma \to 0^+} \int_0^\delta \partial_\sigma h(\sigma,x) q(x) \, \dd x  \, .
\end{eqnarray}
Moreover, if there exists $\delta>0$ such that $\vert \partial^4_x q \vert \leq B_\delta < \infty$ on $[0,\delta]$ with $q(0)=0$, we have
\begin{eqnarray}
\Bigg \vert \int_0^\delta \partial_\sigma h(\sigma,x) \left( q(x) - \sum_{n=1}^3 \frac{\partial^n_x q(0)}{n!} x^n \right)  \, \dd x \Bigg \vert 
\leq 
B_\delta    \int_0^\delta   \vert \partial_\sigma h(\sigma,x) \vert x^4   \, \dd x  \, .
\end{eqnarray}
Let us then introduce the integrals
\begin{eqnarray}
I_n(\sigma) &=& \int_0^\infty \partial_\sigma h(\sigma,x) x^n  \, \dd x , \quad 1 \leq n\leq 3 \, ,\\
J_4(\sigma) &=& \int_0^\infty \vert \partial_\sigma h(\sigma,x) \vert x^4 \, \dd x ,
\end{eqnarray}
The latter integral can be evaluated in closed form as
\begin{eqnarray}
\lefteqn{J_4(\sigma)
=
e^{-\sigma/2}(36+\sigma(61+\sigma(16+\sigma))) \sqrt{\frac{\sigma}{2 \pi }} + } \\
&& \hspace{40pt} \sigma(5+\sigma)(15+ \sigma(12+\sigma)) \left(1+\mathrm{Erf} \left(\sqrt{\frac{\sigma}{2}} \right) \right) \, .
\end{eqnarray}
The above expression shows that  $\lim_{\sigma \to 0^+}  J_4(\sigma)=0$ so that
\begin{eqnarray}
\lim_{\sigma \to 0^+}  \int_0^\delta \vert \partial_\sigma h(\sigma,x) \vert x^4   \, \dd x =  \lim_{\sigma \to 0^+}  J_4(\sigma) - \lim_{\sigma \to 0^+}  \int_\delta^\infty \vert \partial_\sigma h(\sigma,x) \vert x^4   \, \dd x = 0 \, .
\end{eqnarray}
This shows that if the limits at stake exist, we must have
\begin{eqnarray}
\lim_{\sigma \to 0^+} \int_0^\delta \partial_\sigma h(\sigma,x)  q(x)  \, \dd x 
&=&
\lim_{\sigma \to 0^+} \sum_{n=1}^3 \frac{\partial^n_x q(0)}{n!} \int_0^\delta \partial_\sigma h(\sigma,x)   x^n  \, \dd x  \, , \\
&=&
\lim_{\sigma \to 0^+} \sum_{n=1}^3 \frac{\partial^n_x q(0)}{n!} I_n(\sigma) \, , 
\end{eqnarray}
where the last equality follows from \eqref{eq:remainBound}.
For $n=1,2,3$, we find that
\begin{eqnarray}
I_1(\sigma) = \frac{e^{-\sigma/2}}{\sqrt{2 \pi \sigma}} + \frac{1}{2} \left(1+\mathrm{Erf} \left(\sqrt{\frac{\sigma}{2}} \right) \right) \, ,
\end{eqnarray}
\begin{eqnarray}
I_2(\sigma)
=
\frac{e^{-\sigma/2}(1+2 \sigma)}{\sqrt{2 \pi \sigma}} + \left( \frac{3}{2} + \sigma\right) \left(1+\mathrm{Erf} \left(\sqrt{\frac{\sigma}{2}} \right) \right) \, ,
\end{eqnarray}
\begin{eqnarray}
I_3(\sigma)
=
3 \left(e^{-\sigma/2}(3+\sigma) \sqrt{\frac{\sigma}{2 \pi }} + \left(\frac{1}{2}+\sigma \left(2+\frac{\sigma}{2} \right) \right)\left(1+\mathrm{Erf} \left(\sqrt{\frac{\sigma}{2}} \right) \right) \right) \, ,
\end{eqnarray}
where one can observe that $I_1(\sigma)$ and $I_2(\sigma)$ diverge when  $\sigma \to 0^+$. 
Such diverging behaviors cancel out under the assumption that $\partial_x q(0)+\partial^2_x q(0)/2=0$, as we then have
\begin{eqnarray}
\sum_{n=1}^3 \frac{\partial^n_x q(0)}{n!} I_n(\sigma) 
= 
\frac{\partial^2_x q(0)}{2} (I_2(\sigma)-I_1(\sigma)) + \frac{\partial^2_x q(0)}{6} I_3(\sigma) \, ,
\end{eqnarray}
with $\lim_{\sigma \to 0^+} I_2(\sigma)-I_1(\sigma) = 1$ and $\lim_{\sigma \to 0^+} I_3(\sigma)=3/2$.
\end{proof}

\end{appendix}

\bibliographystyle{imsart-number}


\end{document}